\documentclass[10pt,draft]{amsart}

\usepackage[cp1251]{inputenc}
\usepackage[T2A]{fontenc}
\usepackage{amsmath}
\usepackage{amsfonts}
\usepackage{amssymb}
\usepackage{amsthm}
\usepackage{euscript}
\usepackage[dvips]{graphicx}
\usepackage[all]{xy}
\usepackage{delarray}
\usepackage{amscd}
\usepackage[ukrainian,russian,english]{babel}

\date{}

\theoremstyle{plain}
\newtheorem{theorem}{Theorem}[section]
\newtheorem{lemma}[theorem]{Lemma}
\newtheorem{proposition}[theorem]{Proposition}
\newtheorem{corollary}[theorem]{Corollary}
\newtheorem*{remark}{Remark}

\large \numberwithin{equation}{section}

\title[Uniform estimates for the semi-periodic eigenvalues]
{Uniform estimates for the semi-periodic eigenvalues of the
singular differential operators}
\address{Institute of Mathematics NAS of Ukraine \\
    Tereshchenkivska str., 3 \\
    Kyiv\\
    Ukraine\\
    01601}
\author[V. A. Mikhailets and V. M. Molyboga]
       {Volodymyr A. Mikhailets and Volodymyr M. Molyboga}
\email{mikhailets@imath.kiev.ua} \email{molyboga@imath.kiev.ua}

\begin{document}
\begin{abstract}
Let $m\in \mathbb{N}$, $\alpha\in[0,1]$, and $V$ be a 1-periodic
complex-valued distribution in the negative Sobolev space
$H^{-m\alpha}[0,1]$. The singular non-self-adjoint eigenvalue
problem $D^{2m}u+V u=\lambda u$, $D=-i d/dx$, with semi-periodic
boundary conditions is investigated. The uniform in $V$ asymptotic
and non-asymptotic eigenvalue estimates are found and proved. The
case of periodic boundary conditions was earlier studied by
authors in the papers \cite{R5,R6}.
\end{abstract}
\maketitle
\section{Introduction and main results}\label{Int}
Consider the eigenvalue problem on the interval $[0,1]$
\begin{equation*}\label{eq_Int10}
  D^{2m}u(x)+V(x)u(x)=\lambda u(x),\quad D=-i d/dx
\end{equation*}
with semi-periodic boundary conditions. Here $V(x)$ is a
1-periodic complex-valued distribution in the negative Sobolev
space $H^{-m\alpha}[0,1]$ with
\begin{equation*}\label{eq_Int12}
  m\in \mathbb{N},\quad \alpha\in[0,1].
\end{equation*}
To investigate the problem we associate with one an unbounded
linear operator $L$ in an appropriate Hilbert space and after that
we study a spectrum of the operator $L$.

If $V(x)$ belongs to the Hilbert space $L^{2}[0,1]$ then the
differential expression
\begin{equation*}\label{eq_Int14}
 l[\cdot]:=D^{2m}+V(x)
\end{equation*}
is regular and semi-periodic boundary conditions
\begin{equation*}\label{eq_Int16}
  u^{(k)}(0)=-u^{(k)}(1),\quad k\in\{0,1,\ldots,2m-1\}
\end{equation*}
are regular in the Birkhoff sense. In this case there exists the
unbounded linear operator $L$ in the Hilbert space $L^{2}[0,1]$
with the dense domain
\begin{equation*}\label{eq_Int18}
  Dom(L)=\left\{u\in H^{2m}[0,1]\left| u^{(k)}(0)=-u^{(k)}(1), k=0,1,\ldots,2m-1\right.\right\}
\end{equation*}
such that
\begin{equation*}\label{eq_Int20}
  L u=l[u],\quad u\in Dom(L).
\end{equation*}
The spectrum $\emph{spec}(L)$ of $L$ is discrete and consists of a
sequence of eigenvalues $\{\lambda_{k}\}_{k\geq 1}$ with the
property  $Re\lambda_{n}\rightarrow\infty$ for
$n\rightarrow\infty$, where the eigenvalues $\lambda_{n}$ are
enumerated with there algebraic multiplicities and ordered
lexicographically so that
\begin{equation*}\label{eq_Int22}
    Re\lambda_{k}<Re\lambda_{k+1},\quad \mbox{or} \quad
    Re\lambda_{k}=Re\lambda_{k+1}\quad and \quad Im\lambda_{k}\leq
    Im\lambda_{k+1}.
\end{equation*}
An asymptotic behaviour of the eigenvalues of $L$ in this case was
investigated earlier in detail (see \cite{R8} and references
therein). It has the following form:
\begin{equation*}
  \lambda_{2n-1},\lambda_{2n}=(2n-1)^{2m}\pi^{2m}+O(n^{2m-3/2}),\quad
  n\rightarrow\infty,
\end{equation*}
since semi-periodic boundary conditions are not strongly regular
\cite{R8}. This general asymptotic formula contains two power
terms. We will prove below that in this situation
\begin{equation*}\label{eq_Int24}
  \lambda_{2n-1},\lambda_{2n}=(2n-1)^{2m}\pi^{2m}+\widehat{V}(0)
  \pm \sqrt{\widehat{V}\left(-2(2n-1)\right)
  \widehat{V}\left(2(2n-1)\right)}+o(n^{-m/2}),\quad
  n\rightarrow\infty,
\end{equation*}
where $\hat{V}(k)$ denote the Fourier coefficients of $V(x)$. The
last formula contains $2m+1$ power terms and in general non-power
term
\begin{equation*}\label{eq_Int26}
  \pm \sqrt{\widehat{V}\left(-2(2n-1)\right)
  \widehat{V}\left(2(2n-1)\right)}\in l^{2}(\mathbb{N}).
\end{equation*}
The aim of this article to study the semi-periodic eigenvalue
problem in a singular case when $V(x)$ is a 1-periodic
complex-valued distribution in the negative Sobolev space
$H^{-m\alpha}[0,1]$. To do it we will consider the problem in the
negative Sobolev space $H_{-}^{-m}[0,1]$ of semi-periodic
distributions. Then the operator $L\equiv L_{m}(V)$ has the
natural domain
\begin{equation*}\label{eq_Int28}
  Dom(L)=H_{-}^{m}[0,1].
\end{equation*}
Here we use the following notation. The complex Sobolev spaces
$H_{+}^{s}[0,1]$, $s\in \mathbb{R}$, of 1-periodic functions or
distributions are defined by means their Fourier coefficients:
\begin{equation*}\label{eq_Int30}
  H_{+}^{s}[0,1]:=\left\{f=\sum_{k\in\mathbb{Z}}\widehat{f}(2k)e^{i 2k\pi
  x}\left|\;\parallel
  f\parallel_{H_{+}^{s}[0,1]}<\infty\right.\right\},
\end{equation*}
where
\begin{align*}\label{eq_Int32}
  \parallel
  f\parallel_{H_{+}^{s}[0,1]} & :=\left(\sum_{k\in\mathbb{Z}}
  \langle 2k\rangle^{2s}\mid\widehat{f}(2k)\mid ^{2}\right)^{1/2},\quad \langle
  k\rangle:=1+|k|, \\
  \widehat{f}(2k) & :=\langle f,e^{i 2k\pi x}\rangle, \quad k\in
  \mathbb{Z}.
\end{align*}
The brackets denote the sesquilinear  pairing between dual spaces
$H_{+}^{s}[0,1]$ and $H_{+}^{-s}[0,1]$ extending the
$L^{2}[0,1]$-inner product
\begin{equation*}\label{eq_Int34}
  \langle f,g\rangle:=\int_{0}^{1}f(x)\overline{g(x)}\,dx, \quad
  f,g\in L^{2}[0,1].
\end{equation*}
In the same fashion the complex Sobolev spaces $H_{-}^{s}[0,1]$,
$s\in \mathbb{R}$, of semi-periodic functions or distributions are
introduced:
\begin{equation*}\label{eq_Int36}
  H_{-}^{s}[0,1]:=\left\{f=\sum_{k\in\mathbb{Z}}\widehat{f}(2k+1)e^{i (2k+1)\pi
  x}\left|\;\parallel
  f\parallel_{H_{-}^{s}[0,1]}<\infty\right.\right\},
\end{equation*}
where
\begin{align*}\label{eq_Int38}
  \parallel
  f\parallel_{H_{-}^{s}[0,1]} & :=\left(\sum_{k\in\mathbb{Z}}
  \langle 2k+1\rangle^{2s}\mid\widehat{f}(2k+1)\mid ^{2}\right)^{1/2},\quad \langle
  k\rangle:=1+|k|, \\
  \widehat{f}(2k+1) & :=\langle f,e^{i (2k+1)\pi x}\rangle, \quad k\in
  \mathbb{Z}.
\end{align*}
Here the brackets denote the sesquilinear  pairing between dual
spaces $H_{-}^{s}[0,1]$ and $H_{-}^{-s}[0,1]$ extending the
$L^{2}[0,1]$-inner product.

Obviously that
\begin{equation*}\label{eq_Int40}
  H_{+}^{0}[0,1]=H_{-}^{0}[0,1]=L^{2}[0,1].
\end{equation*}
Also we use the weighted $l^{2}$-spaces
\begin{equation*}
 h^{s,n}\equiv h^{s,n}(\mathbb{Z};\mathbb{C})
\end{equation*}
for any $n\in\mathbb{Z}$ and $s\in\mathbb{R}$. These spaces are
the Hilbert spaces of sequences $(a(k))_{k\in \mathbb{Z}}$ in
$\mathbb{C}$ with the norm
\begin{equation*}\label{eq_Int42}
  \parallel a\parallel_{h^{s,n}}:=\left(\sum_{k\in\mathbb{Z}}
  \langle k+n\rangle^{2s}|a(k)| ^{2}\right)^{1/2}.
\end{equation*}
For n=0 we will simply write $h^{s}$ instead of $h^{s,o}$. To
shorten notation, it is convenient to denote by $h^{s}(n)$ the
$n$-th element of a sequence $(a(k))_{k\in \mathbb{Z}}$ in
$h^{s}$. It is clear that if $a\in h^{s}$, then
\begin{equation*}\label{eq_Int44}
  a(n)=o(|n|^{-s}),\quad |n|\rightarrow \infty.
\end{equation*}
The following Theorem summaries the main results of this paper.
\begin{theorem}\label{th_Int10}
Let $V\in H_{+}^{-m\alpha}[0,1]$, $m\in \mathbb{N}$, $\alpha\in
[0,1]$, and $R>0$.
\begin{itemize}
    \item [(1)] Let $\alpha=1$.
\begin{itemize}
  \item [(a)] There exists $\varepsilon>0$ such that for any $W\in H_{+}^{-m}[0,1]$
  with
\begin{equation*}\label{eq_Int46}
    \|W-V\|_{H_{+}^{-m}[0,1]}\leq \varepsilon
\end{equation*}
the eigenvalues of the operator $L_{m}(W)$ satisfy the asymptotic
formulae
\begin{equation*}\label{eq_Int48}
    \lambda_{2n-1}(m,W),\lambda_{2n}(m,W)=(2n-1)^{2m}\pi^{2m}+O(n^{m}),\quad
    n\rightarrow\infty
\end{equation*}
uniformly in $W$.
  \item [(b)] For any $W\in H_{+}^{-m}[0,1]$ with
\begin{equation*}\label{eq_Int50}
    \|W-V\|_{H_{+}^{m}[0,1]}\leq R
\end{equation*}
the eigenvalues of the operator $L_{m}(W)$ satisfy the asymptotic
formulae
\begin{equation*}\label{eq_Int52}
    \lambda_{2n-1}(m,W),\lambda_{2n}(m,W)=(2n-1)^{2m}\pi^{2m}+o(n^{m}),\quad
    n\rightarrow\infty
\end{equation*}
uniformly in $W$.
\end{itemize}
    \item [(2)] Let $\alpha\in [1/2,1)$. For any $V\in
    H_{+}^{-m\alpha}[0,1]$ with
    \begin{equation*}\label{eq_Int54}
    \|V\|_{H_{+}^{-m\alpha}[0,1]}\leq R
    \end{equation*}
    the eigenvalues of the operator $L_{m}(V)$ satisfy the asymptotic formulae
    \begin{align*}\label{eq_Int56}
    \lambda_{2n-1}(m,V),\lambda_{2n}(m,V)
    & =(2n-1)^{2m}\pi^{2m}+\widehat{V}(0) \\
    & \pm \sqrt{\widehat{V}\left(-2(2n-1)\right)
    \widehat{V}\left(2(2n-1)\right)}+h^{m(1-2\alpha)-\varepsilon}(n),
    \end{align*}
    uniformly in $V$.
    \item [(3)] Let $\alpha\in [0,1/2)$. For any $V\in
    H_{+}^{-m\alpha}[0,1]$ with
    \begin{equation*}\label{eq_Int58}
    \|V\|_{H_{+}^{-m\alpha}[0,1]}\leq R
    \end{equation*}
    the eigenvalues of the operator $L_{m}(V)$ satisfy the asymptotic formulae
    \begin{align*}
    \lambda_{2n-1}(m,V),\lambda_{2n}(m,V)
    & =(2n-1)^{2m}\pi^{2m}+\widehat{V}(0) \\
    & \pm \sqrt{\widehat{V}\left(-2(2n-1)\right)
    \widehat{V}\left(2(2n-1)\right)}+h^{m(1/2-\alpha)}(n),
    \end{align*}
    uniformly in $V$.
\end{itemize}
\end{theorem}
If the distribution $V(x)$ is \textit{real-valued} and
$\alpha\in[0,1/2)$ then our asymptotic formulae are of the form
\begin{align*}
 \lambda_{2n-1}(m,V) & =(2n-1)^{2m}\pi^{2m}+\widehat{V}(0)
 -|\widehat{V}\left(4n-2\right)|+h^{m(1/2-\alpha)}(n), \\
 \lambda_{2n}(m,V) & =(2n-1)^{2m}\pi^{2m}+\widehat{V}(0)
 +|\widehat{V}\left(4n-2\right)|+h^{m(1/2-\alpha)}(n)
\end{align*}
and in the case $m=1$, $\alpha=0$ reproduce the Marchenko's
estimates \cite{R4}. They turn out to be uniform on bounded sets
of $V\in L^{2}[0,1]$.

An improved version of the eigenvalue estimates for
$m\in\mathbb{N}$, $\alpha\in[0,1/2)$ is given in Section 5. The
remainder terms in these estimates are in
$h^{m(1-2\alpha)-\varepsilon}$, $\varepsilon>0$.

Also we prove in Section 3 and Section 6 the following
non-asymptotic estimates for the eigenvalues.
\begin{theorem}\label{th_Int12}
Let $V\in H_{+}^{-m\alpha}[0,1]$, $m\in \mathbb{N}$, $\alpha\in
[0,1]$, $C>1$ and $R>0$.
\begin{itemize}
    \item [(1)] Let $\alpha=1$. There exist $\varepsilon>0$,
$M\geq 1$ and $n_{0}\in \mathbb{N}$ such that for any $W\in
H_{+}^{-m}[0,1]$ with
\begin{equation*}\label{eq_326}
    \|W-V\|_{H_{+}^{-m}[0,1]}\leq\varepsilon
\end{equation*}
for all $n>n_0$ the estimates:
     \begin{align*}\label{eq_328}
        |\lambda_{2n-1}(m,W)-(2n-1)^{2m}\pi^{2m}| & <(2n-1)^{m}, \\
        |\lambda_{2n}(m,W)-(2n-1)^{2m}\pi^{2m}| & <(2n-1)^{m}.
     \end{align*}
are hold.
    \item [(2)] Let $\alpha\in [0,1)$. There exist $M=M(R)\geq 1$ and
    $n_{0}=n_{0}(R,C)\in\mathbb{N}$ such that for any $V\in
    H_{+}^{-m\alpha}[0,1]$ with
    \begin{equation*}\label{eq_322}
    \|V\|_{H_{+}^{-m\alpha}[0,1]}\leq R
    \end{equation*}
for all $n>n_0$ the estimates:
     \begin{align*}\label{eq_324}
       |\lambda_{2n-1}(m,V)-(2n-1)^{2m}\pi^{2m}| & <3^{m}\sqrt{2}C R(2n-1)^{m\alpha}, \\
        |\lambda_{2n}(m,V)-(2n-1)^{2m}\pi^{2m}| & <3^{m}\sqrt{2}C R(2n-1)^{m\alpha}.
     \end{align*}
are hold. The constant $n_{0}$ is efficient.
\end{itemize}
\end{theorem}
The similar results for the periodic eigenvalues were proved in
the papers \cite{R2,R7} ($m=1$) and \cite{R5,R6} ($m\geq 1$).
\section{The spectral problem in the Hilbert sequence
space}\label{Mproblem} In this Section we introduce and study the
matrix operator $T$ in the Hilbert sequence space which is unitary
equivalent to the differential operator $L$ and has the same
spectrum.

Further we denote by $h_{+}^{s,n}\equiv
h_{+}^{s,n}(\mathbb{Z};\mathbb{C})$ and $h_{-}^{s,n}\equiv
h_{-}^{s,n}(\mathbb{Z};\mathbb{C})$ the subspaces of
$h^{s,n}(\mathbb{Z};\mathbb{C})$ defined by
\begin{align*}
    & h_{+}^{s,n}:=\{a\in h^{s,n}| a(2k+1)=0, \forall
    k\in\mathbb{Z}\}, \\
    & h_{-}^{s,n}:=\{a\in h^{s,n}| a(2k)=0, \forall
    k\in\mathbb{Z}\}.
\end{align*}
And also we denote by $$h_{+,0}^{s,n}\equiv
h_{+,0}^{s,n}(\mathbb{Z};\mathbb{C})$$ the subspace of
$h_{+}^{s,n}(\mathbb{Z};\mathbb{C})$ defined by
\begin{equation*}\label{eq_Pr10}
  h_{+,0}^{s,n}:=\{a\in h_{+}^{s,n}| a(0)=0\}.
\end{equation*}
Obviously that
\begin{equation*}\label{eq_Pr12}
  h^{s,n}=h_{+}^{s,n}\oplus h_{-}^{s,n},\quad s\in\mathbb{R},
  n\in\mathbb{Z}.
\end{equation*}
The map
\begin{equation*}
  f\longmapsto(\widehat{f}(2k))_{k\in \mathbb{Z}}
\end{equation*}
is an isometric isomorphism of the space $H^{s}_{+}[0,1]$ onto
$h_{+}^{s}$, and the map
\begin{equation*}
  g\longmapsto(\widehat{g}(2k+1))_{k\in \mathbb{Z}}
\end{equation*}
is an isometric isomorphism of the space $H^{s}_{-}[0,1]$ onto
$h_{-}^{s}$, $s\in \mathbb{R}$.

For these isomorphisms the multiplication of functions corresponds
to convolution of sequences, where the convolution product of two
sequences $$a=(a(k))_{k\in \mathbb{Z}},\quad b=(b(k))_{k\in
\mathbb{Z}}$$ (formally) defined as the sequence given by
\begin{equation}\label{eq_Pr14}
    (a*b)(k):=\sum_{j\in\mathbb{Z}}a(k-j)b(j).
\end{equation}
So, given two functions $f$, $g$ formally,
\begin{equation}\label{eq_Pr16}
    (\widehat{f\cdot g})(k)=\sum_{j\in\mathbb{Z}}\hat{f}(k-j)\hat{g}(j).
\end{equation}
The following Convolution Lemma is the modification of the Main
Convolution Lemma \cite{R2} and very important for our method.
\begin{lemma}[Convolution Lemma]\label{l_10}
Let $n\in\mathbb{Z}$, $s,r\geq 0$, and $t\in\mathbb{R}$ with
$t\leq\min(s,r)$. If $s+r-t>1/2$, than the convolution map is
continuous (uniformly in n), when viewed as a map
\begin{align*}
  & (a')\;h_{+}^{r,n}\times h_{-}^{s,-n}\longrightarrow h_{-}^{t},\quad
  & (a'')&\;h_{+}^{r,n}\times h_{+}^{s,-n}\longrightarrow h_{+}^{t},\qquad
  & (a''')&\;h_{-}^{r,n}\times h_{-}^{s,-n}\longrightarrow h_{+}^{t},\\
  & (b')\;h_{+}^{-t}\times h_{-}^{s,n}\longrightarrow h_{-}^{-r,n},\quad
  & (b'')&\;h_{+}^{-t}\times h_{+}^{s,n}\longrightarrow h_{+}^{-r,n},\qquad
  & (b''')&\;h_{-}^{-t}\times h_{-}^{s,n}\longrightarrow h_{+}^{-r,n},\\
  & (c')\;h_{+}^{t}\times h_{-}^{-s,n}\longrightarrow h_{-}^{-r,n},\quad
  & (c'')&\;h_{+}^{t}\times h_{+}^{-s,n}\longrightarrow h_{+}^{-r,n},\qquad
  & (c''')&\;h_{-}^{t}\times h_{-}^{-s,n}\longrightarrow h_{+}^{-r,n}.
\end{align*}
\end{lemma}
So, the maps
\begin{align}
     H^{-m\alpha}_{+}[0,1]\times H^{m(2-\alpha)}_{-}[0,1]&\longmapsto
      H^{-m\alpha}_{-}[0,1],\quad &(V,f)&\longmapsto V\cdot f, \\
     H^{m(2-\alpha)}_{+}[0,1]\times H^{-m\alpha}_{-}[0,1]&\longmapsto
      H^{-m\alpha}_{-}[0,1],\quad &(V,f)&\longmapsto V\cdot f
\end{align}
are continuous, when $V\cdot f$ is given by formula
\eqref{eq_Pr16}.

For a distribution
\begin{equation*}\label{eq_Pr18}
  f=\sum_{k\in\mathbb{Z}}\hat{f}(2k)e^{i 2k\pi x}
\end{equation*}
we can define the conjugate distribution
\begin{equation*}\label{eq_Pr20}
  \bar{f}=\sum_{k\in\mathbb{Z}}\overline{\hat{f}(-2k)}e^{i 2k\pi x}.
\end{equation*}
A distribution $f\in H^{s}_{+}[0,1]$ is said to be real-valued if
$\bar{f}=f$, i.e. the corresponding sequence of the Fourier
coefficients is Hermitian-symmetric:
\begin{equation*}\label{eq_Pr22}
  \overline{\hat{f}(2k)}=\hat{f}(-2k),\quad \forall k\in\mathbb{Z}.
\end{equation*}
Let $m\in \mathbb{N}$, $\alpha\in [0,1]$, and $v$ be in
$h_{+}^{-m\alpha}$. Consider in the Hilbert sequence space
$h_{-}^{-m\alpha}$ the unbounded linear operator
\begin{equation*}\label{eq_Pr24}
  T_{\alpha}\equiv T_{\alpha}(v):=A^{m}+B(v),\quad T_{1}\equiv T
\end{equation*}
with the dense domain
\begin{equation*}\label{eq_Pr26}
  Dom(T_{\alpha})=h_{-}^{m(2-\alpha)},
\end{equation*}
where $A^{m}$ and $B(v)$ are the infinite matrices,
\begin{align*}
  A(2k-1,2j-1):&=(2k-1)^{2}\pi^{2}\delta_{kj},\quad
  &A(2k,2j):&=0, \\
  A^{m}(2k-1,2j-1)&=(2k-1)^{2m}\pi^{2m}\delta_{kj},\quad &A^{m}(2k,2j)&=0,\quad k,j\in \mathbb{Z}
\end{align*}
and
\begin{equation*}\label{eq_Pr28}
  B(v)(2k-1,2j-1):=v(2k-2j),\quad B(v)(2k,2j):=0, \quad k,j\in \mathbb{Z}.
\end{equation*}
Obviously, that the operator $A^{m}$ is a positive self-adjoint
operator in the Hilbert sequence space $h_{-}^{-m\alpha}$ with the
dense domain
\begin{equation*}\label{eq_Pr30}
  Dom(A^{m})=h_{-}^{m(2-\alpha)}.
\end{equation*}
The spectrum of $A^{m}$ is discrete:
\begin{equation*}\label{eq_Pr32}
  \emph{spec}(A^{m})=\{(2k-1)^{2m}\pi^{2m} \left|\right. k\in
  \mathbb{N}\},
\end{equation*}
where all eigenvalues are double.
\begin{lemma}\label{l_12}
The operator $B(v)$, $v\in h_{+}^{-m\alpha}$ with the domain
\begin{equation*}\label{eq_Pr34}
  Dom(B(v))=h_{-}^{m(2-\alpha)}
\end{equation*}
is $A^{m}$-bounded, and its relative bound is equal 0.
\end{lemma}
\begin{proof}
According to the Convolution Lemma there exist the constants
$C_{\alpha,m}^{(1)}>0$ and $C_{\alpha,m}^{(2)}>0$ such that
\begin{equation*}
  \parallel B(v) u\parallel_{h_{-}^{-m\alpha}}\leq
  \begin{cases}
    C_{\alpha,m}^{(1)}\parallel v \parallel_{h_{+}^{m(2-\alpha)}}\parallel u\parallel_{h_{-}^{-m\alpha}},
    & v\in h_{+}^{m(2-\alpha)}, u\in h_{-}^{-m\alpha}, \\
    C_{\alpha,m}^{(2)}\parallel v \parallel_{h_{+}^{-m\alpha}}\parallel u\parallel_{h_{-}^{m(2-\alpha)}},
    & v\in h_{+}^{-m\alpha}, u\in h_{-}^{m(2-\alpha)}.
  \end{cases}
\end{equation*}
Further, for any fixed $\delta>0$ there exists a decomposition
$$v=v_{0}+v_{\delta}$$ with
\begin{equation*}
  v_{0}\in h_{+}^{m(2-\alpha)},\; v_{\delta}\in h_{+}^{-m\alpha},\; \parallel v_{\delta}
  \parallel_{h_{+}^{-m\alpha}}<\frac{\delta}{C_{\alpha,m}^{(2)}}.
\end{equation*}
Taking into account that
\begin{equation*}
  \parallel u\parallel_{h_{-}^{m(2-\alpha)}}\leq \parallel u\parallel_{h_{-}^{-m\alpha}}+\parallel
  A^{m} u\parallel_{h_{-}^{-m\alpha}}, \quad u\in h_{-}^{m(2-\alpha)}
\end{equation*}
then we have the following estimates:
\begin{align*}
 \parallel B(v) u\parallel_{h_{-}^{-m\alpha}} & \leq\parallel B(v_{0})u\parallel_{h_{-}^{-m\alpha}}+\parallel
  B(v_{\delta})u\parallel_{h_{-}^{-m\alpha}} \\
  & \leq  C_{\alpha,m}^{(1)}\parallel v_{0}
  \parallel_{h_{+}^{m(2-\alpha)}}\parallel u\parallel_{h_{-}^{-m\alpha}}
   + C_{\alpha,m}^{(2)}\parallel v_{\delta} \parallel_{h_{+}^{-m\alpha}}\parallel
   u\parallel_{h_{-}^{m(2-\alpha)}}\\
  & \leq\delta\parallel A^{m} u\parallel_{h_{-}^{-m\alpha}}
  +\left(C_{\alpha,m}^{(2)}\parallel v_{0}
  \parallel_{h_{+}^{m(2-\alpha)}}+\delta\right)\parallel
  u\parallel_{h_{-}^{-m\alpha}}.
\end{align*}
Hence $B(v)\ll A^{m}$.
\end{proof}
\begin{corollary}\label{cor_10}
The operator $B(v)$ is form-bounded with respect to operator
$A^{m}$ and its relative bound is equal 0.
\end{corollary}
\begin{corollary}\label{cor_12}
The operator $T_{\alpha}$ is quasi-sectorial. More precisely, for
any $\varepsilon>0$ there exists $c_{\varepsilon}>0$ such that for
any $f\in Dom(T_{\alpha})$
\begin{equation*}
  \left|\arg\left((T_{\alpha}+c_{\varepsilon}Id)f,f\right)_{h_{-}^{-m\alpha}}\right|\leq
  \varepsilon.
\end{equation*}
\end{corollary}
\begin{proposition}\label{pr_10}
Let $m\in \mathbb{N}$, $\alpha\in [0,1]$, and $v$ be in
$h_{+}^{-m\alpha}$.
\begin{enumerate}
  \item [(1)] The operator $T_{\alpha}$ is quasi-$m$-sectorial.
  \item [(3)] A resolvent set of the operator $T_{\alpha}$ is not
    empty and its resolvent $R(\lambda, T_{\alpha})$ is a
    compact operator.
\end{enumerate}
\end{proposition}
\begin{proof}
Let prove the Proposition.
\begin{enumerate}
  \item [(1)] The operator $T_{\alpha}$ is quasi-sectorial. Its
  maximality property follows from Statement~(3) of this
  Proposition.
  \item [(3)] For any $\lambda\in\mathbb{C}$
  with $Re\lambda\leq 0$ the following estimates are valid:
  \begin{align*}
    \|A^{m}R(\lambda,A^{m})\|_{\mathcal{L}(h_{-}^{-m\alpha})}
    &\leq\frac{2^{2m}\pi^{2m}}{|\lambda-\pi^{2m}|}, \\
    \|R(\lambda,A^{m})\|_{\mathcal{L}(h_{-}^{-m\alpha})}
    &\leq\frac{1}{|\lambda-\pi^{2m}|}.
  \end{align*}
  Therefore the formulated statement follows from Theorem 3.17 (\cite{R3}, Ch. IV)
  and the inequalities since the resolvent $R(\lambda,A^{m})$ is a compact operator.
\end{enumerate}
\end{proof}
\begin{corollary}\label{cor_14}
The spectrum $\emph{spec}(T_{\alpha})$ of the operator
$T_{\alpha}$ is discrete and consists of a sequence of the
eigenvalues
\begin{equation*}
    \lambda_{k}=\lambda_{k}(\alpha,m,v), \quad k\in\mathbb{N}
\end{equation*}
with the property that
\begin{equation*}
    Re\lambda_{k}\rightarrow +\infty\quad \mbox{as} \quad k
    \rightarrow +\infty,
\end{equation*}
where the eigenvalues $\lambda_{k}$ are enumerated with their
algebraic multiplicities ordered lexicographically.
\end{corollary}

If $v\in h_{+}^{-m\alpha}$, then for any $\beta$ such that
$$0\leq\alpha\leq\beta\leq 1,$$ the operator $T_{\beta}$ is well
defined and its spectrum is discrete.
\begin{proposition}\label{pr_12}
If $\alpha$ and $\beta$ as above, then
\begin{equation*}
  \emph{spec}(T_{\beta})=\emph{spec}(T_{\alpha}).
\end{equation*}
\end{proposition}
\begin{proof}
Obviously that
\begin{equation*}
  \emph{spec}(T_{\beta})\supseteq\emph{spec}(T_{\alpha})
\end{equation*}
since
\begin{equation*}
  T_{\beta}\supseteq T_{\alpha}.
\end{equation*}
To prove the converse inclusion for the spectra sufficient to show
that any eigenvector (or root vector) $f\in
Dom(T_{\beta})=h_{-}^{-m\beta}$ of $T_{\beta}$ is in fact an
element of $Dom(T_{\alpha})=h_{-}^{-m\alpha}$. So, let $\lambda\in
\emph{spec}(T_{\beta})$, and
\begin{equation*}
  (T_{\beta}-\lambda Id)f=g, \quad f,g\in Dom(T_{\beta}),
\end{equation*}
where $f$ is eigenvector if $g=0$, and root vector if $g\neq 0$.
Taking into account that $$B(v)f\in h_{-}^{-m\alpha}$$ by the
Convolution Lemma, we conclude that
\begin{equation*}
  A^{m}f=g+\lambda f-B(v)f\in h_{-}^{-m\alpha}.
\end{equation*}
Since $$(A^{m})^{-1}\in
\mathcal{L}(h_{-}^{-m\alpha},h_{-}^{m(2-\alpha)}),$$ we have got
\begin{equation*}
  f=(A^{m})^{-1}A^{m}f\in h_{-}^{m(2-\alpha)}
\end{equation*}
as claimed. This establish that
\begin{equation*}
  \emph{spec}(T_{\beta})=\emph{spec}(T_{\alpha}),
\end{equation*}
and the proof is complete.
\end{proof}
To study the eigenvalue problem
\begin{equation*}
  T u=\lambda u
\end{equation*}
we will compare the spectrum $\emph{spec}(T_{\alpha})$ of the
operator
\begin{equation*}
  T_{\alpha}=A^{m}+B(v)
\end{equation*}
with the spectrum of the unperturbed operator $A^{m}$ in the same
space,
\begin{equation*}
  \emph{spec}(A^{m})=\{(2k-1)^{2m}\pi^{2m} \left|\right. k\in
  \mathbb{N}\}.
\end{equation*}
And take into account that for any $\alpha\in[0,1]$ the operator
$T$ is isospectral to the operator $T_{\alpha}$.

Further, for given $M\geq 1$, $n\geq 1$, and
$0<r_{n}<(2n-1)^{m}\pi^{2m}$ the following regions $Ext_{M}$ and
$Vert^{m}_{n}(r_{n})$ of complex plane will be used:
\begin{align*}\label{eq_064}
     Ext_{M}&:=\left\{\lambda\in\mathbb{C}\left| Re\,\lambda\leq|Im\,\lambda|-M\right.\right\},\\
     Vert^{m}_{n}(r_{n})&:=\left\{\lambda=(2n-1)^{2m}\pi^{2m}+z\in\mathbb{C}\left| |Re\,z| \leq
     (2n-1)^{m}\pi^{2m}, |z|\geq r_{n}\right.\right\}.
\end{align*}

Adding a constant to a convolution operator $B(v)$ results in a
shift of the spectrum of the operator
\begin{equation*}
  T_{\alpha}=A^{m}+B(v)
\end{equation*}
by the same constant. Therefore we assume bellow, without loss of
generality, that
\begin{equation*}
  v(0)=0.
\end{equation*}
It is clear that
\begin{equation*}
  |v(0)|\leq \|v\|_{h_{+}^{-m\alpha}}.
\end{equation*}

\section{Non-asymptotic estimates}\label{model_problem}
In this Section we will prove some non-asymptotic estimates for
the eigenvalues of $\emph{spec}(T)=\emph{spec}(T_{\alpha})$. For
this purpose let decompose the operator
\begin{equation*}
  \lambda-A^{m}-B(v)
\end{equation*}
in the following way. For $\lambda\in\mathbb{C}
\backslash\emph{spec}(A^{m})$ write
\begin{equation*}
  \lambda-A^{m}-B(v)=A^{m/2}_{\lambda}(I_{\lambda}-S_{\lambda})A^{m/2}_{\lambda},
\end{equation*}
where $A^{m/2}_{\lambda}$, $I_{\lambda}$ and $S_{\lambda}$ are the
following infinite matrices $(k,j\in\mathbb{Z})$
\begin{align*}
  A_{\lambda}^{m}(2k-1,2j-1)&:=|\lambda-(2k-1)^{2m}\pi^{2m}|\delta_{kj}, & A_{\lambda}^{m}(2k,2j)&=0, \\
  I_{\lambda}(2k-1,2j-1)&:=\frac{\lambda-(2k-1)^{2m}\pi^{2m}}{|\lambda-(2k-1)^{2m}\pi^{2m}|}\delta_{kj},  & I_{\lambda}(2k,2j)&=0, \\
  S_{\lambda}(2k-1,2j-1)&:=\frac{v(2k-2j)}{|\lambda-(2k-1)^{2m}\pi^{2m}|^{1/2}|\lambda-(2j-1)^{2m}\pi^{2m}|^{1/2}},
  & S_{\lambda}(2k,2j)&=0.
\end{align*}
Note that $A^{m/2}_{\lambda}$ and $I_{\lambda}$ are diagonal
matrices independent on $v$. Both $I_{\lambda}$ and $S_{\lambda}$
can be viewed as linear operators on $h_{-}^{0}$. The reason for
working with $I_{\lambda}$ instead of the identity matrix $Id$ is
that we want to avoid having to take complex square roots in the
definitions of $A^{m/2}_{\lambda}$ and $S_{\lambda}$, \cite{R7}.
Clearly, for any $t,s\in\mathbb{R}$ with $s-t\leq 1$, and any
$\lambda\in\mathbb{C} \backslash\emph{spec}(A^{m})$, we have
$A^{-m/2}_{\lambda}\in \mathcal{L}(h_{-}^{mt},h_{-}^{ms})$ with
norm
\begin{equation}\label{eq_500}
  \|A^{-m/2}_{\lambda}\|_{\mathcal{L}(h_{-}^{mt},h_{-}^{ms})}=
  \sup_{k\in\mathbb{Z}}\frac{\langle 2k-1\rangle^{m(s-t)}}{\mid\lambda-(2k-1)^{2m}\pi^{2m}\mid^{1/2}}<\infty.
\end{equation}
Further, it is clearly that any
$\lambda\in\mathbb{C}\backslash\emph{spec}(A^{m})$ with
$\|S_{\lambda}\|_{\mathcal{L}(h_{-}^{0})}<1$ is in the resolvent
set $\emph{Resol}(T_{\alpha})$ of $$T_{\alpha}=A^{m}+B(v),$$ and
the resolvent operator is of the form
\begin{equation}\label{eq_502}
  (\lambda-A^{m}-B(v))^{-1}=A^{-m/2}_{\lambda}(I_{\lambda}-S_{\lambda})^{-1}A^{-m/2}_{\lambda},
\end{equation}
where the right side of \eqref{eq_502} is viewed as a composition
\begin{equation*}
  h^{-m\alpha}\rightarrow h^{0}\rightarrow
  h^{0}\rightarrow h^{m}(\hookrightarrow h^{-m\alpha}).
\end{equation*}
In a straightforward way one can prove
\begin{lemma}\label{l_20}
  Let $m\in\mathbb{N}$, $\alpha\in[0,1]$, $M\geq 1$, and $v\in
  h_{+,0}^{-m\alpha}$. Then, for any $\lambda\in Ext_{M}$,
  $$
    \|S_{\lambda}\|_{\mathcal{L}(h_{-}^{0})}\leq 2^{2m+1}\|v\|_{h_{+}^{-m\alpha}}
    \frac{1}{M^{(1-\alpha)/2+1/4}}.
  $$
\end{lemma}
\begin{proof} We estimate the $\mathcal{L}(h_{-}^{0})$-norm of $S_{\lambda}$ by its
Hilbert-Schmidt norm,
\begin{equation}\label{eq_504}
  \|S_{\lambda}\|_{\mathcal{L}(h_{-}^{0})}\leq\left(\sum_{k,j}\frac{|v(2k-2j)|^{2}}{|\lambda-
  (2k-1)^{2m}\pi^{2m}||\lambda-(2j-1)^{2m}\pi^{2m}|}\right)^{1/2}.
\end{equation}
Using
\begin{equation*}
  \langle k-j\rangle^{2m\alpha}\leq 2^{2m\alpha}(\langle k\rangle^{2m\alpha}+\langle
  j\rangle^{2m\alpha}),\quad(k,j\in\mathbb{Z})
\end{equation*}
together with the trigonometric estimate
\begin{equation*}
  |\lambda-k^{2m}\pi^{2m}|\geq(M+k^{2m}\pi^{2m})\sin(\frac{\pi}{4}),\quad (k\in\mathbb{Z},
  \lambda\in Ext_{M}),
\end{equation*}
one concludes
\begin{align*}
    \|S_{\lambda}\|_{\mathcal{L}(h_{-}^{0})}
   & \leq\left(\sum_{k,j} \frac{2\cdot 2^{2m\alpha}(\langle 2k-1\rangle^{2m\alpha}+\langle 2j-1\rangle^{2m\alpha})}{(M+(2k-1)^{2m}\pi^{2m})
     (M+(2j-1)^{2m}\pi^{2m})}\langle 2k-2j\rangle^{-2m\alpha}|v(2k-2j)|^{2}\right)^{1/2} \\
   & \leq\left(4\cdot 2^{2m\alpha}\sup_{k}\frac{\langle 2k-1\rangle^{2m\alpha}}{(M+(2k-1)^{2m}\pi^{2m})}\sum_{k}\frac{1}
     {(M+(2k-1)^{2m}\pi^{2m})}\right)^{1/2} \\
   & \cdot\left(\sum_{j}\langle 2k-2j\rangle^{-2m\alpha}|v(2k-2j)|^{2}\right)^{1/2} \\
   & =2^{m\alpha+1}\left(\sup_{k}\frac{\langle 2k-1\rangle^{2m\alpha}}{(M+(2k-1)^{2m}\pi^{2m})}\right)^{1/2}
     \left(\sum_{k}\frac{1}{M+(2k-1)^{2m}\pi^{2m}}\right)^{1/2}\|v\|_{h_{+}^{-m\alpha}}.
\end{align*}
\end{proof}
For $\lambda\in Vert^{m}_{n}(r_{n})$, the following estimate for
$\|S_{\lambda}\|_{\mathcal{L}(h_{-}^{0})}$ can be obtained:
\begin{lemma}\label{l_22}
Let $m\in\mathbb{N}$, $\alpha\in[0,1]$, $n\geq
\frac{8m^{2}+4m-7}{2(8m-7)}$, $0<r_{n}<(2n-1)^{m}\pi^{2m}$, and
$v\in h_{+,0}^{-m\alpha}$. Then, for any $\lambda\in
Vert^{m}_{n}(r_{n})$,
\begin{align*}
   \|S_{\lambda}\|_{\mathcal{L}(h_{-}^{0})}&\leq\frac{1}{r_{n}}\left(|v(2(2n-1))|+|v(-2(2n-1))|\right) \\
    &+4\left(\frac{2}{\pi}\right)^{m}\left(\frac{(2n-1)^{m(\alpha-1+1/2m)}}{\sqrt{r_{n}}}+
    \frac{6\log (2n-1)}{(2n-1)^{m(1-\alpha)}}\right)\|v\|_{h_{+}^{-m\alpha}}.
\end{align*}
\end{lemma}
To prove Lemma \ref{l_22}, one uses that for $\lambda\in
Vert^{m}_{n}(r_{n}), n\geq\frac{8m^{2}+4m-7}{2(8m-7)}, k\neq\pm
(2n-1)$
\begin{equation}\label{eq_506}
  \frac{1}{|\lambda-k^{2m}\pi^{2m}|}\leq\frac{3}{\pi^{2m}}
  \frac{1}{|k^{2m}-(2n-1)^{2m}|},
\end{equation}
together with the following elementary estimates:
\begin{lemma}\label{l_24}
Let $m\in\mathbb{N}$, $\alpha\in[0,1]$, and $n\geq m$. Then
\begin{itemize}
  \item [(a)] $\sup_{k\neq\pm
n}\frac{<k>^{m\alpha}}{|k^{2m}-n^{2m}|^{1/2}}\leq
3^{m\alpha}n^{m(\alpha-1+\frac{1}{2m})}$;
  \item [(b)] $\sup_{k\neq\pm n}\frac{<k\pm
n>^{m\alpha}}{|k^{2m}-n^{2m}|^{1/2}}\leq
4^{m\alpha}n^{m(\alpha-1+\frac{1}{2m})}$;
  \item [(c)] $\sum_{k\neq\pm n}\frac{1}{|k^{2m}-n^{2m}|^{1/2}}\leq 5\frac{1+\log n}{n}$.
\end{itemize}
\end{lemma}
Lemma \ref{l_22} together with the estimate
\begin{equation*}
  \left(|v(2(2n-1))|+|v(-2(2n-1))|\right)\leq 3^{m}\sqrt{2}\|v\|_{h_{+}^{-m\alpha}}(2n-1)^{m\alpha}
\end{equation*}
leads to
\begin{equation*}
  \|S_{\lambda}\|_{\mathcal{L}(h_{-}^{0})}\leq\frac{3^{m}\sqrt{2}(2n-1)^{m\alpha}\|v\|_{h_{+}^{-m\alpha}}}
  {r_{n}}+4\left(\frac{2}{\pi}\right)^{m}\left(\frac{(2n-1)^{m(\alpha-1+1/2m)}}{\sqrt{r_{n}}}+
  \frac{6\log (2n-1)}{(2n-1)^{m(1-\alpha)}}\right)\|v\|_{h_{+}^{-m\alpha}}.
\end{equation*}
Combining this with Lemma \ref{l_20} one obtains
\begin{proposition}\label{pr_14}
Let $m\in\mathbb{N}$, $\alpha\in[0,1)$, $\mathrm{R}>0$,
$\mathrm{C}> 1$, and $r_{n}:=3^{m}\sqrt{2}
\mathrm{C}\mathrm{R}(2n-1)^{m\alpha}$ $(n\geq1)$. Then there exist
$\mathrm{M}=\mathrm{M}(R)\geq 1$ and $n_{0}=n_{0}(R,C)
\geq\frac{8m^{2}+4m-7}{2(8m-7)}$ with $0<r_{n}<(2n-1)^{m}\pi^{2m}$
$(n\geq n_{0})$ so that, for any $v\in h_{+,0}^{-m\alpha}$ with
$\|v\|_{h_{+}^{-m\alpha}}\leq \mathrm{R}$
\begin{equation*}
  \|S_{\lambda}\|_{\mathcal{L}(h_{-}^{0})}<1\quad\mbox{for}\quad\lambda\in Ext_{M}
  \cup\bigcup_{n\geq n_{0}}Vert^{m}_{n}(r_{n}).
\end{equation*}
Hence
\begin{equation*}
  Ext_{M}\cup\bigcup_{n\geq
  n_{0}}Vert^{m}_{n}(r_{n})\subseteq\emph{Resol}(T_{\alpha}).
\end{equation*}
\end{proposition}
So, the spectrum $\emph{spec}(T_{\alpha}(v))$ of $T_{\alpha}(v)$
is contained in the complement of the set
$$Ext_{M}\cup\bigcup_{n\geq n_{0}}Vert^{m}_{n}(r_{n}).$$ To
localize the eigenvalues notice that for any $0\leq s\leq1$ in
fact
\begin{equation*}
  Ext_{M}\cup\bigcup_{n\geq
  n_{0}}Vert^{m}_{n}(r_{n})\subseteq\emph{Resol}(T_{\alpha}(sv)).
\end{equation*}
Hence, for any contour $$\Gamma\subseteq Ext_{M}\cup\bigcup_{n\geq
n_{0}}Vert^{m}_{n}(r_{n})$$ and any $0\leq s\leq 1$, the Riesz
projector
\begin{equation*}
  P(s):=\frac{1}{2\pi
  i}\int_{\Gamma}(\lambda-A^{m}-B(sv))^{-1}\,d\lambda\in\mathcal{L}(h^{-m\alpha})
\end{equation*}
is well defined and depends continuously on $s$. Since projectors
whose difference has small norm have isomorphic ranges (see, e.g.
\cite{R1,R3}), continuity of the map
\begin{equation*}
  s\mapsto P(s)
\end{equation*}
implies that the dimension of the range $P(s)$ is independent of
$s$. Therefore the number of eigenvalues of $A^{m}+B(v)$ and
$A^{m}$ inside $\Gamma$ (counted with their algebraic
multiplicities) are the same.

Summarizing the result above, we obtain the following statement.
\begin{theorem}\label{pr_16}
Let $m\in\mathbb{N}$, $\alpha\in[0,1)$, $\mathrm{C}>1$, and
$\mathrm{R}>0$. Then there exist $\mathrm{M}=\mathrm{M}(R)\geq 1$
and $n_{0}=n_{0}(R,C)\in\mathbb{N}$ so that, for any $v\in
h_{+,0}^{-m\alpha}$ with $$\|v\|_{h_{+}^{-m\alpha}}\leq R,$$ the
spectrum $\emph{spec}(T_{\alpha}(v))$ of $T_{\alpha}(v)$ satisfies
the estimates:
      \begin{itemize}
        \item [(a)] There are precisely $2n_0$ eigenvalues inside
    the bounded cone
    $$
    T_{M,n_0}=
    \left\{\lambda\in\mathbb{C}\,\left|\,|Im\;\lambda|-M\leq Re\;\lambda \leq
    \left((2n_{0})^{2m}-(2n_{0})^{m}\right)\pi^{2m}\right.\right\}.
    $$
        \item [(b)] For any $n>n_0$ the pairs of eigenvalues $\lambda_{2n-1}(\alpha,m,v)$,
     $\lambda_{2n}(\alpha,m,v)$ are inside a disc around $(2n-1)^{2m}\pi^{2m}$:
     \begin{align*}\label{eq_324}
       |\lambda_{2n-1}(\alpha,m,v)-(2n-1)^{2m}\pi^{2m}| & <3^{m}\sqrt{2}C R(2n-1)^{m\alpha}, \\
        |\lambda_{2n}(\alpha,m,v)-(2n-1)^{2m}\pi^{2m}| & <3^{m}\sqrt{2}C R(2n-1)^{m\alpha}.
     \end{align*}
    \end{itemize}
\end{theorem}
So, by Theorem \ref{pr_16}, it follows that uniformly for bounded
sets of sequence $v$ in $h_{+}^{-m\alpha}$ the one-term asymptotic
formulae
\begin{equation}\label{eq_508}
  \lambda_{2n-1}(\alpha,m,v),\lambda_{2n}(\alpha,m,v)=(2n-1)^{2m}\pi^{2m}+O(n^{m\alpha}).
\end{equation}
are hold.

In the next Sections we improve the asymptotic formulae
\eqref{eq_508}. For this purpose we will consider vertical strips
$Vert^{m}_{n}(r_{n})$ with a circle of radius $r_{n}=(2n-1)^{m}$
around $(2n-1)^{2m}\pi^{2m}$ removed, and the following estimate
for the operators $S_{\lambda}$ will be useful:
\begin{lemma}\label{l_26}
Let $m\in\mathbb{N}$, $\alpha\in[0,1)$, and $\varepsilon>0$. Then
there exists $C=C(\alpha,m,\varepsilon)$ such that, for any $v\in
h_{+,0}^{-m\alpha}$
\begin{equation*}
   \left\|\left(\sup_{\lambda\in Vert^{m}_{n}((2n-1)^{m})}\|S_{\lambda}\|
    _{\mathcal{L}(h_{-}^{0})}\right)_{n\geq 1}\right\|_{h^{m(1-\alpha-\varepsilon)}}\leq
    C\|v\|_{h_{+}^{-m\alpha}}.
\end{equation*}
\end{lemma}
\begin{proof}
For any given $n\geq 1$, split the operator $S_{\lambda}$,
$$S_{\lambda}=\sum_{j=1}^{6}I_{A_{n}^{(j)}}S_{\lambda},$$ where
$$I_{A}:\mathbb{Z}\times\mathbb{Z}\rightarrow\mathbb{R}$$ denotes
the characteristic function of a set
$$A\subseteq\mathbb{Z}\times\mathbb{Z}$$ and $A^{(j)}\equiv
A^{(j)}_{n}$  is the following decomposition of
$\mathbb{Z}\times\mathbb{Z}$
\begin{align*}
 A^{(1)}&:=\{(k,j)\in\mathbb{Z}^{2}\,\left|\right.\,2k-1,2j-1\in\{\pm (2n-1)\}\};\\ A^{(2)}&:=\{(k,j)\in\mathbb{Z}^{2}\,\left|\right.
  \,k,j\in\mathbb{Z}\backslash\{\pm (2n-1)\}\};\\
 A^{(3)}&:=\{(k,j)\in\mathbb{Z}^{2}\,\left|\right.\,k=n,j\neq\pm (2n-1)\};\\  A^{(4)}&:=\{(k,j)\in\mathbb{Z}^{2}\,\left|\right.
  \,2k-1=-(2n-1),j\neq\pm (2n-1)\}; \\
  A^{(5)}&:=\{(k,j)\in\mathbb{Z}^{2}\,\left|\right.\,2k-1\neq\pm (2n-1),2j-1=2n-1\}; \\ A^{(6)}&:=\{(k,j)\in\mathbb{Z}^{2}\,\left|\right.
  \,2k-1\neq\pm (2n-1),2j-1=-(2n-1)\}.
\end{align*}
Then
\begin{equation*}
  \sup_{\lambda\in Vert^{m}_{n}((2n-1)^{m})}\|S_{\lambda}\|_{\mathcal{L}(h_{-}^{0})}
  \leq\sum_{j=1}^{6}\sup_{\lambda\in Vert^{m}_{n}((2n-1)^{m})}
  \|I_{A_{n}^{(j)}}S_{\lambda}\|_{\mathcal{L}(h_{-}^{0})}
\end{equation*}
and each term in the latter sum is treated separately.

As $v(0)=0$, we have, for any $\lambda\in
Vert^{m}_{n}((2n-1)^{m})$
\begin{align*}
   \|I_{A_{n}^{(1)}}S_{\lambda}\|_{\mathcal{L}(h_{-}^{0})}&\leq
    \left(\sum_{(2k-1,2j-1)=\pm (2n-1,-2n+1)}\frac{|v(2k-2j)|^{2}}{|\lambda-
    (2k-1)^{2m}\pi^{2m}||\lambda-(2j-1)^{2m}\pi^{2m}|}\right)^{1/2} \\
  & \leq\frac{1}{(2n-1)^{m}}\left(|v(2(2n-1))|^{2}+|v(-2(2n-1))|^{2}\right)^{1/2}\in h^{m(1-\alpha)}.
\end{align*}
The operators $I_{A_{n}^{(j)}}S_{\lambda}$ for $3\leq j\leq 6$ are
estimated similar. Let us consider e.g. the case $j=3$. For
non-negative sequences $a$ in $h^{0}$ with $a(n)=0$, $\forall
n\leq 0$ and $b,c$ in $h_{-}^{0}$, we have
\begin{align*}
 & \sum_{n}a(n)\langle  n\rangle^{m(1-\alpha)}\sup_{\lambda\in \mathrm{Vert^{m}_{n}((2n-1)^{m})}}
  \left|\sum_{k}\left(I_{A_{n}^{(3)}}S_{m\lambda}b\right)(2k-1)
  \cdot c(2k-1)\right|  \\
 & \leq\sum_{n\geq 1}\sum_{2k-1=2n-1,2j-1\neq\pm (2n-1)}a(2n-1)\langle  2n-1\rangle^{m(1-\alpha)} \\
 &\cdot\sup_{\lambda\in \mathrm{Vert^{m}_{n}((2n-1)^{m})}}\frac{|v(2k-2j)|}{|\lambda-
   (2k-1)^{2m}\pi^{2m}|^{1/2}|\lambda-(2j-1)^{2m}\pi^{2m}|^{1/2}}\,b(2j-1)\,c(2k-1) \\
 & \leq\frac{\sqrt{3}}{\pi^{m}}\sup_{n\geq1,j\neq\pm (2n-1)}\frac{\langle  2n-1\rangle^{m(1-\alpha)}\langle  2n-2j\rangle ^{m\alpha}}
 {(2n-1)^{m/2}|(2j-1)^{2m}-(2n-1)^{2m}|^{1/2}} \\
 &\cdot\sum_{n\geq 1,2j-1\neq\pm
 (2n-1)}\frac{|v(2n-2j)|}{\langle  2n-2j\rangle^{m\alpha}}\,a(2n-1)\,b(2j-1)\,c(2n-1),
\end{align*}
where for the last inequality we use \eqref{eq_506}.

By the Cauchy-Schwartz inequality and the following estimate
obtained from Lemma \ref{l_24}(b)
\begin{equation*}
 \sup_{n\geq1,j\neq\pm n}\frac{\langle n\rangle^{m(1-\alpha)}\langle n-j\rangle^{m\alpha}}
 {n^{m/2}|j^{2m}-n^{2m}|^{1/2}}\leq 4^{m}n^{(-1/2+1/2m)}
\end{equation*}
one gets
\begin{align*}
\sum_{n}a(n)\langle n\rangle^{m(1-\alpha)}&\sup_{\lambda\in
Vert^{m}_{n}((2n-1)^{m})}\left|\sum_{k}\left(I_{A_{n}^{(3)}}S_{\lambda}b\right)(2k-1)\cdot
c(2k-1)\right| \\
&\leq\frac{4^{m}\sqrt{3}}{\pi^{m}}\|v\|_{h_{+}^{-m\alpha}}\|a\|_{h_{-}^{0}}
\|b\|_{h_{-}^{0}}\|c\|_{h_{-}^{0}}.
\end{align*}

It remains to estimate $\|I_{A_{n}^{(2)}}
S_{\lambda}\|_{\mathcal{L}(h_{-}^{0})}$. For non-negative
sequences $a$ in $h^{0}$ with $a(n)=0$, $\forall n\leq 0$ and
$b,c$ in $h_{-}^{0}$, and any $\varepsilon>0$ (without loss of
generality we assume $1-\alpha-\varepsilon\geq 0$), one obtains in
the same fashion as above
\begin{align*}
    &\sum_{n}a(n)\langle n\rangle^{m(1-\alpha-\varepsilon)}\sup_{\lambda\in \mathrm{Vert^{m}_{n}((2n-1)^{m})}}
     \left|\sum_{k}\left(I_{A_{n}^{(2)}}S_{m\lambda}b\right)(2k-1)\cdot
     c(2k-1)\right| \\
   & \leq\frac{{3}}{\pi^{2m}}\sum_{n\geq1,2k-1,2j-1\neq\pm (2n-1)}R_{m}(n,k,j)
   \frac{|v(2k-2j)|}{\langle 2k-2j\rangle^{m\alpha}}\,a(n)\,b(2j-1)\,c(2k-1),
\end{align*}
where
\begin{equation*}
  R_{m}(n,k,j):=\frac{\langle n\rangle^{m(1-\alpha-\varepsilon)}\langle 2k-2j\rangle^{m\alpha}}
  {|(2k-1)^{2m}-n^{2m}|^{1/2}|(2j-1)^{2m}-n^{2m}|^{1/2}}.
\end{equation*}
The latter sum is estimated using Cauchy-Schwartz inequality. To
estimate $R_{m}(n,k,j)$, we split up
$A_{n}^{(2)}=\{(2k-1,2j-1)\in\mathbb{Z}^{2}\,\left|\right.
\,2k-1,2j-1\neq\pm (2n-1)\}$. First notice that, as $R_{m}(n,k,j)$
is symmetric in $k$ and $j$, it suffices to consider the case
$|2j-1|\leq |2k-1|$. Then, using $$\langle
k-j\rangle^{m\alpha}\leq2^{m\alpha}\langle k\rangle^{m\alpha}$$
and $$\langle n\rangle^{m(1-\alpha-\varepsilon)}\leq
2^{m(1-\alpha-\varepsilon)} n^{m(1-\alpha-\varepsilon)},$$ we have
got
\begin{equation*}
  R_{m}(n,k,j)\leq 2^{m}\frac{n^{m(1-\alpha-\varepsilon)}\langle 2k-1\rangle^{m\alpha}}
  {|(2k-1)^{2m}-n^{2m}|^{1/2}|(2j-1)^{2m}-n^{2m}|^{1/2}}.
\end{equation*}
For the subsets $A_{n}^{(\pm,\pm)}\cap\{|2j-1|\leq |2k-1|\leq
2n\}$ of $A_{n}^{(2)}$, $$
  A_{n}^{(\pm,\pm)}:=\{(2k-1,2j-1)\in A_{n}^{(2)}\,\left|\right.\,\pm (2k-1)\geq 0;\;\pm
  (2j-1)\geq 0\},
$$ we argue similarly. Consider e.g. $A_{n}^{(-,+)}$. Then
$|2k-1|+n=|2k-1-n|$ and $|2j-1|+n=|2j-1+n|$, hence $$
  n^{m(1-\alpha-\varepsilon)}\langle 2k-1\rangle^{m\alpha}\leq n^{m(1-\alpha-\varepsilon)}
  (1+2n)^{m\alpha}\leq 3^{m\alpha}n^{m(1-\varepsilon)}:
$$ 1) $m=2l+1,\;l\in\mathbb{N}$
\begin{align*}
n^{m(1-\alpha-\varepsilon)}\langle 2k-1\rangle^{m\alpha}&\leq
3^{m\alpha}n^{m(1-\varepsilon)}\leq 3^{m\alpha}
|(2k-1)^{m}-n^{m}|^{(1-\varepsilon)/2}|(2j-1)^{m}+n^{m}|^{(1-\varepsilon)/2}
\\
& \leq
3^{m\alpha}|(2k-1)^{m}-n^{m}|^{1/2}|(2k-1)^{m}+n^{m}|^{-\varepsilon/2}
\\
&
\cdot|(2j-1)^{m}+n^{m}|^{1/2}|(2j-1)^{m}-n^{m}|^{-\varepsilon/2},
\end{align*}
which leads to $$
  R_{m}(n,k,j)\leq
  6^{m}|(2k-1)^{m}+n^{m}|^{-(1+\varepsilon)/2}|(2j-1)^{m}-n^{m}|^{-(1+\varepsilon)/2};
$$ 2) $m=2l,\;l\in\mathbb{N}$
\begin{align*}
n^{m(1-\alpha-\varepsilon)}\langle 2k-1\rangle^{m\alpha}&\leq
3^{m\alpha}n^{m(1-\varepsilon)}\leq
3^{m\alpha}|(2k-1)^{m}+n^{m}|^{(1-\varepsilon)/2}|(2j-1)^{m}+n^{m}|^{(1-\varepsilon)/2}
\\
   & \leq
   3^{m\alpha}|(2k-1)^{m}+n^{m}|^{1/2}|(2k-1)^{m}-n^{m}|^{-\varepsilon/2}
   \\
   & |(2j-1)^{m}+n^{m}|^{1/2}|(2j-1)^{m}-n^{m}|^{-\varepsilon/2},
\end{align*}
which leads to $$
  R_{m}(n,k,j)\leq
  6^{m}|(2k-1)^{m}-n^{m}|^{-(1+\varepsilon)/2}|(2j-1)^{m}-n^{m}|^{-(1+\varepsilon)/2}.
$$ Therefore $$
  R_{m}(n,k,j)\leq
  6^{m}|(-1)^{m+1}(2k-1)^{m}+n^{m}|^{-(1+\varepsilon)/2}|(2j-1)^{m}-n^{m}|^{-(1+\varepsilon)/2}.
$$ By the Cauchy-Schwartz inequality one then gets
\begin{align}\label{eq_510}
  & \sum_{n\geq 1}\sum_{2k-1,2j-1\neq\pm (2n-1)}|(-1)^{m+1}(2k-1)^{m}-n^{m}|^{-(1+\varepsilon)}
    |(2j-1)^{m}-n^{m}|^{-(1+\varepsilon)}\notag \\
  & \cdot\frac{|v(2k-2j)|}{\langle 2k-2j\rangle^{m\alpha}}\,a(n)\,b(2j-1)\,c(2k-1) \notag
    \\
  & \leq\left(\sum_{n\geq 1}a^{2}(n)\sum_{2j-1\neq\pm (2n-1)}|(2j-1)^{m}-n^{m}|^{-(1+\varepsilon)}
    \sum_{2k-1\neq\pm (2n-1)}\langle 2k-2j\rangle^{-2m\alpha}|v(2k-2j)|^{2}\right)^{1/2} \\
  & \cdot\left(\sum_{j}b^{2}(2j-1)\sum_{k}c^{2}(2k-1)
   \sum_{n\geq 1,2k-1\neq -(2n-1)}|(2k-1)^{m}-n^{m}|^{-(1+\varepsilon)}\right)^{1/2}\notag \\
  & \leq C\|v\|_{h_{+}^{-m\alpha}}\|a\|_{h^{0}}\|b\|_{h_{-}^{0}}\|c\|_{h_{-}^{0}}\notag.
\end{align}
Next consider the subsets $A_{n}^{(\pm,\pm)}\cap\{|2j-|\leq
|2k-1|;\;|2k-1|>2n\}$ of $A_{n}^{(2)}$. Again we argue similarly
for each of these subsets. Consider e.g. $A_{n}^{(+,+)}$. In the
case $\alpha\in[0,1/2)$, choose without loss of generality
$\varepsilon>0$ with $\frac{1}{2}-\alpha-\frac{\varepsilon}{2}\geq
0$. Then
\begin{align*}
     n^{m(1-\alpha-\varepsilon)}\langle 2k-1\rangle^{m\alpha}&\leq 2^{m\alpha}n^{m(1-\varepsilon)/2}
     n^{m(1-2\alpha-\varepsilon)/2}|2k-1|^{m\alpha} \\
   & \leq 2^{m\alpha}|(2j-1)^{m}+n^{m}|^{(1-\varepsilon)/2}|(2k-1)^{m}+n^{m}|^{(1-2\alpha-\varepsilon)/2}
     |(2k-1)^{m}+n^{m}|^{\alpha} \\
   & \leq
   2^{m\alpha}|(2k-1)^{m}-n^{m}|^{1/2}|(2k-1)^{m}+n^{m}|^{-\varepsilon/2}\\
   & \cdot|(2j-1)^{m}+n^{m}|^{1/2}|(2j-1)^{m}-n^{m}|^{-\varepsilon/2}
\end{align*}
and we gets $$
  R_{m}(n,k,j)\leq
  4^{m}|(2k-1)^{m}-n^{m}|^{-(1+\varepsilon)/2}|(2j-1)^{m}-n^{m}|^{-(1+\varepsilon)/2},
$$ and thus obtain estimate of the type \eqref{eq_510}.

In the case $\alpha\in[1/2,1)$, since $n\leq|2k-1+n|$, $$
  n^{m(1-\alpha-\varepsilon)}\langle 2k-1\rangle^{m\alpha}\leq 2^{m\alpha}|(2k+1)^{m}+n^{m}|^{(1-\alpha-
  \varepsilon)}|2k-1|^{m\alpha}
$$ so using that
$|k|^{m(\alpha-1/2)}\leq|(2k-1)^{m}+n^{m}|^{(\alpha-1/2)}$ and
$|k|^{m/2)}\leq 3^{m/2}|(2k-1)^{m}-n^{m}|$, we get
\begin{align*}
    n^{m(1-\alpha-\varepsilon)}\langle 2k-1\rangle^{m\alpha}&\leq 2^{m\alpha}|(2k-1)^{m}+n^{m}|^{(1-\alpha-
    \varepsilon)}|2k-1|^{m\alpha} \\
  & \leq 2^{m\alpha}|(2k-1)^{m}+n^{m}|^{(1-\alpha-\varepsilon)}3^{m/2}|(2k-1)^{m}-n^{m}|^{1/2}
    |(2k-1)^{m}+n^{m}|^{(\alpha-1/2)} \\
  & \leq(2\sqrt{3})^{m}|k^{2m}-n^{2m}|^{1/2}|k^{m}+n^{m}|^{-\varepsilon} \\
  & \leq(2\sqrt{3})^{m}|(2k-1)^{2m}-n^{2m}|^{1/2}|(2j-1)^{m}-n^{m}|^{-\varepsilon/2}
    |(2j-1)^{m}+n^{m}|^{-\varepsilon/2},
\end{align*}
where we use that $|(2j-1)^{m}\pm n^{m}|\leq|(2k-1)^{m}+n^{m}|$.
This yields $$
  R_{m}(n,k,j)\leq
  8^{m}|(2j-1)^{m}+n^{m}|^{-(1+\varepsilon)/2}|(2j-1)^{m}-n^{m}|^{-(1+\varepsilon)/2}
$$ and therefore we again obtain an estimate of the type
\eqref{eq_510}.
\end{proof}
For later reference, let us denote for given $m\in\mathbb{N}$,
$\alpha\in[0,1)$ and $R>0$, by $n_{*}=n_{*}(\alpha,m,R)\geq 1$ a
number with the property that, for any $v\in h_{+,0}^{-m\alpha}$
with $$\|v\|_{h_{+}^{-m\alpha}}\leq R$$ we have got
\begin{equation}\label{eq_512}
  \sup_{\lambda\in Vert^{m}_{n}((2n-1)^{m})}\|S_{\lambda}\|
  _{\mathcal{L}(h_{-}^{0})}\leq\frac{1}{2},\quad \forall n\geq n_{*}.
\end{equation}

\section{Asymptotic estimates of $\tau_{n}$}
In this and the next Sections we establish asymptotic estimates
for the eigenvalues $\lambda_{n}$. They are obtained by separately
considering the mean $\tau_{n}$ and the difference $\gamma_{n}$ of
a pair $\lambda_{2n}$ and $\lambda_{2n-1}$,
\begin{equation*}
  \tau_{n}:=\frac{\lambda_{2n}+\lambda_{2n-1}}{2},\quad
  \gamma_{n}:=\lambda_{2n}-\lambda_{2n-1}.
\end{equation*}

In this Section we establish
\begin{theorem}\label{pr_18}
Let $m\in\mathbb{N}$, $\alpha\in[0,1)$, and $\varepsilon>0$. Then,
uniformly for bounded sets of $v$ in $h_{+,0}^{-m\alpha}$,
\begin{equation*}
  \tau_{n}=(2n-1)^{2m}\pi^{2m}+h^{m(1-2\alpha-\varepsilon)}(n).
\end{equation*}
\end{theorem}
The assertion of Theorem \ref{pr_18} is a consequence of Lemma
\ref{l_28} and Lemma \ref{l_30} below. Let $R>0$ and $v\in
h_{+,0}^{-m\alpha}$ with $$\|v\|_{h_{+}^{-m\alpha}}\leq R.$$ For
$n\geq n_{*}=n_{*}(\alpha,m,R)$ with $n_{*}$ chosen as in
\eqref{eq_512}, define the Riesz projectors
\begin{align*}
    P_{n}&:=\frac{1}{2\pi i}\int_{\Gamma_{n}}(\lambda-A^{m}-B(v))^{-1}\,d\lambda\in\mathcal{L}(h_{-}^{-m\alpha}), \\
    P_{n}^{0}&:=\frac{1}{2\pi i}\int_{\Gamma_{n}}(\lambda-A^{m})^{-1}\,d\lambda\in\mathcal{L}(h_{-}^{-m\alpha}),
\end{align*}
where $\Gamma_{n}$ is the positively oriented contour given by
\begin{equation*}
 \Gamma_{n}=\{\lambda\in\mathbb{C}\,|\,|\lambda-(2n-1)^{2m}\pi^{2m}|=(2n-1)^{m}\}.
\end{equation*}
The corresponding Riesz spaces are the ranges of these projectors,
\begin{equation*}
  E_{n}:=P_{n}(h_{-}^{-m\alpha}),\quad\text{and}\quad
  E_{n}^{0}:=P_{n}^{0}(h_{-}^{-m\alpha}).
\end{equation*}
Both $E_{n}$ and $E_{n}^{0}$ are two-dimensional subspaces of
$h_{-}^{0}$, and $P_{n}$ as well as $(A^{m}+B(v))P_{n}$ can be
considered as operators from $\mathcal{L}(h_{-}^{0})$. Their
traces can be computed to be
\begin{equation*}
  Tr(P_{n})=2,\quad Tr((A^{m}+B(v))P_{n})=2\tau_{n}.
\end{equation*}
Similarly, we have $$
  Tr(P_{n}^{0})=2,\quad Tr(A^{m}P_{n}^{0})=2(2n-1)^{2m}\pi^{2m}
$$ and thus obtain $$
  2\tau_{n}-2(2n-1)^{2m}\pi^{2m}=Tr((A^{m}+B(v))P_{n})-Tr(A^{m}P_{n}^{0})=Tr(Q_{n}),
$$ where $Q_{n}$ is the operator $$
  Q_{n}:=(A^{m}+B(v)-(2n-1)^{2m}\pi^{2m})P_{n}-(A^{m}-(2n-1)^{2m}\pi^{2m})P_{n}^{0}\in\mathcal{L}(h_{-}^{0}).
$$ Substituting the formula for $P_{n}$ and $P_{n}^{0}$ one gets
\begin{align*}
  Q_{n}& =\frac{1}{2\pi i}\int_{\Gamma_{n}}(\lambda-(2n-1)^{2m}\pi^{2m})
          \left((\lambda-A^{m}-B(v))^{-1}-(\lambda-A^{m})^{-1}\right)\,d\lambda \\
          & =\frac{1}{2\pi i}\int_{\Gamma_{n}}(\lambda-(2n-1)^{2m}\pi^{2m})
            (\lambda-A^{m}-B(v))^{-1}B(v)(\lambda-A^{m})^{-1}\,d\lambda\in\mathcal{L}(h_{-}^{0}).
\end{align*}
Remark that
\begin{equation*}
 Q_{n}^{0}(2k,2l)=0,\quad k,l\in \mathbb{Z}.
\end{equation*}
Write $Q_{n}=Q_{n}^{0}+Q_{n}^{1}$ with $$
  Q_{n}^{0}:=\frac{1}{2\pi i}\int_{\Gamma_{n}}(\lambda-(2n-1)^{2m}\pi^{2m})
             (\lambda-A^{m})^{-1}B(v)(\lambda-A^{m})^{-1}\,d\lambda\in\mathcal{L}(h_{-}^{0}),
$$ which leads to the following expression for $\tau_{n}$,
\begin{equation}\label{eq_513}
    \tau_{n}=(2n-1)^{2m}\pi^{2m}+\frac{1}{2}Tr(Q_{n}^{0})+\frac{1}{2}Tr(Q_{n}^{1}).
\end{equation}
To compute
\begin{equation*}
  Tr(Q_{n}^{0})=\sum_{k\in\mathbb{Z}}Q_{n}^{0}(k,k)=\sum_{k\in\mathbb{Z}}Q_{n}^{0}(2k-1,2k-1)
\end{equation*}
we need the following
\begin{lemma}\label{l_28}
Let $m\in\mathbb{N}$, and $\alpha\in[0,1)$. For any $v\in
h_{+,0}^{-m\alpha}$ with $n\geq 1$ and $k,l\in\mathbb{Z}$,
\begin{align*}
  Q_{n}^{0}(2k-1,2l-1)= &
  \begin{cases}
    v(\pm 2(2n-1))  & \text{if}\quad (k,l)=(n,-n+1)\quad \text{or}\quad (k,l)=(-n+1,n); \\
    0         & \text{otherwise}.
  \end{cases}
\end{align*}
\end{lemma}
\begin{proof}
For $k,l\in\mathbb{Z}$ and $n\geq 1$, we have
\begin{align*}
  Q_{n}^{0}(2k-1,2l-1) & =\frac{1}{2\pi i}\int_{\Gamma_{n}}\frac{\lambda-(2n-1)^{2m}\pi^{2m}}
  {(\lambda-(2k-1)^{2m}\pi^{2m})(\lambda-(2l-1)^{2m}\pi^{2m})}v(2k-2l)\,d\lambda \\
  & =v(2k-2l)\frac{1}{2\pi i}\int_{\Gamma_{n}}\frac{\lambda-(2n-1)^{2m}\pi^{2m}}{(\lambda-(2k-1)^{2m}\pi^{2m})
   (\lambda-(2l-1)^{2m}\pi^{2m})}\,d\lambda
\end{align*}
and the claimed statement follows from
\begin{equation*}
  \frac{1}{2\pi i}\int_{\Gamma_{n}}\frac{\lambda-(2n-1)^{2m}\pi^{2m}}{(\lambda-(2k-1)^{2m}\pi^{2m})
  (\lambda-(2l-1)^{2m}\pi^{2m})}\,d\lambda=
  \begin{cases}
    1 & \text{if}\quad (k,l)\in\{n,-n+1\}, \\
    0 & \text{otherwise}.
  \end{cases}
\end{equation*}
\end{proof}
Lemma \ref{l_28} implies that $$
  Tr(Q_{n}^{0})=0,
$$ and, moreover, $\emph{range}(Q_{n}^{0})\subseteq E_{n}^{0}$.
So, since $\emph{range}(Q_{n})\subseteq\emph{span}(E_{n}\cup
E_{n}^{0})$ by definition, we conclude that
$$\emph{range}(Q_{n}^{1}) \subseteq\emph{span} (E_{n}\cup
E_{n}^{0})$$ as well. Hence $\emph{range}(Q_{n}^{1})$ is at most
dimension four and $$|Tr(Q_{n}^{1})|\leq
4\|Q_{n}^{1}\|_{\mathcal{L}(h_{-}^{0})}.$$ Theorem \ref{pr_18}
then follows from \eqref{eq_513} together with
\begin{lemma}\label{l_30}
Let $m\in\mathbb{N}$,  $\alpha\in[0,1)$, $R>0$ and
$\varepsilon>0$. Then there exists $C=C(\alpha,m,\varepsilon)$ so
that, for any $v\in h_{+,0}^{-m\alpha}$ with
$$\|v\|_{h_{+}^{-m\alpha}}\leq R,$$
\begin{equation*}
  \left\|\left(\|Q_{n}^{1}\|_{\mathcal{L}(h_{-}^{0})}\right)_{n\geq n_{*}}\right\|
  _{h^{m(1-2\alpha-\varepsilon)}}\leq C\|v\|_{h_{+}^{-m\alpha}}^{2},
\end{equation*}
where $n_{*}=n_{*}(\alpha,m,R)$ is given by \eqref{eq_512}.
\end{lemma}
\begin{proof}
By \eqref{eq_502}, for any $\lambda\in\Gamma_{n}$,
$(\lambda-A^{m}-B(v))^{-1}$ is given by
$A^{-m/2}_{\lambda}(I_{\lambda}-S_{\lambda})^{-1}A^{-m/2}_{\lambda}$.
Hence
\begin{align*}
  Q_{n}^{1} & =\frac{1}{2\pi i}\int_{\Gamma_{n}}(\lambda-(2n-1)^{2m}\pi^{2m})
            \left((\lambda-A^{m}-B(v))^{-1}-(\lambda-A^{m})^{-1}B(v)(\lambda-A^{m})^{-1}
            \right)\,d\lambda \\
              & =\frac{1}{2\pi i}\int_{\Gamma_{n}}(\lambda-(2n-1)^{2m}\pi^{2m})
            A^{-m/2}_{\lambda}(I_{\lambda}-S_{\lambda})^{-1}S_{\lambda}
            I_{\lambda}^{-1}S_{\lambda}I_{\lambda}^{-1}A^{-m/2}_{\lambda}\,d\lambda.
\end{align*}
Using \eqref{eq_500} and \eqref{eq_506} one shows that, for
$\lambda\in Vert^{m}_{n}(r_{n})$
$(n\geq\frac{8m^{2}+4m-7}{2(8m-7)}$,
$0<r_{n}<(2n-1)^{m}\pi^{2m})$,
\begin{equation}\label{eq_514}
  \|A^{-m/2}_{\lambda}\|_{\mathcal{L}(h_{-}^{0})}\leq r_{n}^{-1/2}+\frac{\sqrt{3}}{\pi^{m}}
  (2n-1)^{-m+1/2}.
\end{equation}
Together with Lemma \ref{l_26} the claimed statement then follows.
\end{proof}
\section{Asymptotic estimates of $\gamma_{n}$}
To state the asymptotic of $\gamma_{n}$ let introduce, for $v\in
h_{+,0}^{-m\alpha}$, the sequence
\begin{equation*}
  w:=\frac{1}{\pi^{2m}}\frac{v}{k^{m}}* \frac{v}{k^{m}}\in
  h_{+}^{mt}\quad\text{with}\quad
  w(2n)=\frac{1}{\pi^{2m}}\sum_{k\neq\pm n}\frac{v(n-k)}{(n-k)^{m}}\cdot
  \frac{v(n+k)}{(n+k)^{m}}
\end{equation*}
Note that $\frac{v}{k^{m}}\in h_{+}^{m(1-\alpha)}$. By the
Convolution Lemma, this implies that $w\in h_{+}^{mt}$, where for
$\alpha\in[0,1-1/2m)$ we have got $t=(1-\alpha)$, and, for
$\alpha\in[1-1/2m,1)$, any $t<2(1-\alpha)-1/2m$ can be chosen. In
particular, we can always chose $t>-\alpha$.

Further, let consider the sequence $(l(n))_{n\in\mathbb{Z}}$, such
that
\begin{equation*}
   l(2n):=\frac{1}{\pi^{2m}}\sum_{k\neq\pm n}\frac{v(n-k)}{n^{m}-k^{m}}\cdot
  \frac{v(n+k)}{n^{m}+k^{m}}.
\end{equation*}
We have got $(l(n))_{n\in\mathbb{Z}}\in  h_{+}^{mt}$, where $t$ is
chosen as above, and
\begin{equation}\label{eq_518}
  \|l\|_{h_{+}^{mt}}\leq Const\|w\|_{h_{+}^{mt}}.
\end{equation}

\begin{theorem}\label{pr_20}
Let $m\in\mathbb{N}$, $\alpha\in[0,1)$ and $\varepsilon>0$. Then,
uniformly on bounded sets of $v$ in $h_{+,0}^{-m\alpha}$,
\begin{equation*}
   \left(\min_{\pm}\left|\gamma_{n}\pm2\sqrt{(v+l)(-2(2n-1))
   \cdot(v+l)(2(2n-1))}\right|\right)_{n\geq 1}\in
    h^{m(1-2\alpha-\varepsilon)}.
\end{equation*}
\end{theorem}
\begin{remark}
An asymptotic estimate only involving $v$ but not $l$ is of the
form
\begin{equation*}
  \left(\min_{\pm}\left|\gamma_{n}\pm 2\sqrt{v(-2(2n-1))
  \cdot v(2(2n-1))}\right|\right)_{n\geq 1}\in
  \begin{cases}
    h^{m(1/2-\alpha)} & \text{if}\quad \alpha\in[0,1/2), \\
    h^{m(1-2\alpha-\varepsilon)} & \text{if}\quad
    \alpha\in[1/2,1).
  \end{cases}
\end{equation*}
\end{remark}
\begin{proof}
To prove Theorem \ref{pr_20}, consider for
$n_{*}=n_{*}(\alpha,m,R)$ and $v\in h_{+,0}^{-m\alpha}$ with
$$\|v\|_{h_{+}^{-m\alpha}}\leq R,$$ the restriction $K_{n}$ of
$A^{m}+B(v)-\tau_{n}$ to the Riesz space $E_{n}$, $$
K_{n}:E_{n}\longrightarrow E_{n}.$$ The eigenvalues of $K_{n}$ are
$\pm\frac{\gamma_{n}}{2}$, hence
$$
  \det(K_{n})=-(\frac{\gamma_{n}}{2})^{2}.
$$
We need the following auxiliary result:
\begin{lemma}\label{l_32}
Let $m\in\mathbb{N}$, $\alpha\in[0,1)$, $R>0$ and $\varepsilon>0$.
Then there exists $C>0$ so that, for any $v\in h_{+,0}^{-m\alpha}$
with $$\|v\|_{h_{+}^{-m\alpha}}\leq R$$ we have
\begin{itemize}
  \item [(i)] $\|P_{n}\|_{\mathcal{L}(h_{-}^{0})}\leq C,\quad \forall n\geq n_{*}$;
  \item [(ii)] $\left(\|P_{n}-P_{n}^{0}\|_{\mathcal{L}(h_{-}^{0})}\right)_{n\geq n_*}\in
  h^{m(1-\alpha-\varepsilon)}$.
\end{itemize}
\end{lemma}
\begin{proof}
Recall that, for $n\geq n_{*}$ $$P_{n}=\frac{1}{2\pi
i}\int_{\Gamma_{n}}A^{-m/2}_{\lambda}(I_{\lambda}-S_{\lambda})^{-1}
A^{-m/2}_{\lambda}\,d\lambda$$ and
$$
  P_{n}-P_{n}^{0}=\frac{1}{2\pi i}\int_{\Gamma_{n}}A^{-m/2}_{\lambda}
  (I_{\lambda}-S_{\lambda})^{-1}S_{\lambda}I_{\lambda}^{-1}
  A^{-m/2}_{\lambda}\,d\lambda.
$$
The claimed estimates then follow from \eqref{eq_500} and Lemma
\ref{l_26}.
\end{proof}

Choose, if necessary, $n_{*}$ larger so that
\begin{equation}\label{eq_520}
  \|P_{n}-P_{n}^{0}\|_{\mathcal{L}(h_{-}^{0})}\leq\frac{1}{2},\quad \forall\,n\geq
  n_{*}.
\end{equation}
One verifies easily that $$Q_{n}:=(P_{n}-P_{n}^{0})^{2}$$ commutes
with $P_{n}$ and $P_{n}^{0}$. Hence $Q_{n}$ leaves both Riesz
spaces $E_{n}$ and $E_{n}^{0}$ invariant. The operator $Q_{n}$ is
used to define, for $n\geq n_{*}$, the restriction of the
transformation operator $$(Id-Q_{n})^{-1/2} (P_{n}P_{n}^{0}+
(Id-P_{n})(Id-P_{n}^{0}))$$  to $E_{n}^{0}$ (cf.\cite{R3}),
$$
  U_{n}:=(Id-Q_{n})^{-1/2}P_{n}P_{n}^{0}:E_{n}^{0}\longrightarrow
  E_{n},
$$
where $(Id-Q_{n})^{-1/2}$ is given by the binomial formula
\begin{equation}\label{eq_522}
  (Id-Q_{n})^{-1/2}=\sum_{l\geq 0}
  \left(
  \begin{array}{c}
    -1/2 \\
    l \
  \end{array}
  \right)
  (-Q_{n})^{l}.
\end{equation}
One verifies that $U_{n}$ is invertible with the inverse given by
\begin{equation}\label{eq_524}
  U_{n}^{-1}:=P_{n}^{0}P_{n}(Id-Q_{n})^{-1/2}.
\end{equation}
As a consequence,
$$
  \det(U_{n}^{-1}K_{n}U_{n})=-(\frac{\gamma_{n}}{2})^{2}.
$$
To estimate $\det(U_{n}^{-1}K_{n}U_{n})$, write
$$
  U_{n}^{-1}K_{n}U_{n}=P_{n}^{0}P_{n}K_{n}P_{n}P_{n}^{0}+
  R_{n}^{(1)}+R_{n}^{(2)},
$$
where
$$
  R_{n}^{(1)}:=(U_{n}^{-1}-P_{n}P_{n}^{0})K_{n}P_{n}P_{n}^{0}; \quad
  R_{n}^{(2)}:=U_{n}^{-1}K_{n}(U_{n}-P_{n}P_{n}^{0}).
$$
The term $$P_{n}^{0}P_{n}K_{n}P_{n}P_{n}^{0}=P_{n}^{0}P_{n}
(A^{m}+B(v)-\tau_{n})P_{n}^{0}$$ is split up further,
\begin{align*}
   P_{n}^{0}P_{n} (A^{m}+B(v)-\tau_{n})P_{n}^{0}&=P_{n}^{0}(A^{m}+B(v)-
   \tau_{n})P_{n}^{0}+P_{n}^{0}(P_{n}-P_{n}^{0})(A^{m}+B(v)-\tau_{n})P_{n}^{0} \\
  &
  =P_{n}^{0}B(v)P_{n}^{0}+L_{n}^{(1)}+P_{n}^{0}(P_{n}-P_{n}^{0})B(v)
  P_{n}^{0}+R_{n}^{(3)},
\end{align*}
where $L_{n}^{(1)}$ is a diagonal operator (use
$A^{m}P_{n}=n^{2m}\pi^{2m} P_{n}^{0}$)
$$
  L_{n}^{(1)}:=P_{n}^{0}((2n-1)^{2m}\pi^{2m}-\tau_{n})P_{n}^{0}
$$
and
$$
  R_{n}^{(3)}:=P_{n}^{0}(P_{n}-P_{n}^{0})((2n-1)^{2m}\pi^{2m}-\tau_{n})P_{n}^{0}.
$$
As a $2\times2$ matrix, $B_{n}:=P_{n}^{0}B(v)P_{n}^{0}$ is given
by
\begin{equation*}
    \begin{pmatrix}
    B_{n}(2n-1,2n-1)  & B_{n}(2n-1,-2n+1) \\
    B_{n}(-2n+1,2n-1) & B_{n}(-2n+1,-2n+1)
    \end{pmatrix}
=
  \begin{pmatrix}
    0      & v(2(2n-1)) \\
    v(-2(2n-1)) & 0
  \end{pmatrix}
.
\end{equation*}
To obtain a satisfactory estimate for
$$\det(U_{n}^{-1}K_{n}U_{n}),$$ we  have to substitute an expansion
of $(P_{n}-P_{n}^{0})$ into $P_{n}^{0}(P_{n}-P_{n}^{0})B(v)
P_{n}^{0}$ and split the main term into a diagonal part
$L_{n}^{(2)}$ and an off-diagonal part. Let us explain this in
more detail. Write
\begin{align*}
  P_{n}-P_{n}^{0} & =\frac{1}{2\pi i}\int_{\Gamma_{n}}(\lambda-A^{m})^{-1}
                     B(v)(\lambda-A^{m})^{-1}\,d\lambda \\
                   & +\frac{1}{2\pi
                   i}\int_{\Gamma_{n}}(\lambda-A^{m})^{-1}B(v)
                    (\lambda-A^{m})^{-1}B(v)(\lambda-A^{m}-B(v))^{-1}\,d\lambda,
\end{align*}
which leads to
$$
  P_{n}^{0}(P_{n}-P_{n}^{0})B(v)P_{n}^{0}=\mathcal{S}_{n}+
  R_{n}^{(4)},
$$
where
$$
  R_{n}^{(4)}:=\frac{1}{2\pi i}\int_{\Gamma_{n}}P_{n}^{0}A^{-m/2}_{\lambda}
  I_{\lambda}^{-1}S_{\lambda}I_{\lambda}^{-1}S_{\lambda}(I_{\lambda}-S_{\lambda})^{-1}
  S_{\lambda}A^{m/2}_{\lambda}P_{n}^{0}\,d\lambda
$$
and
\begin{equation}\label{eq_526}
  \mathcal{S}_{n}:= \frac{1}{2\pi
  i}\int_{\Gamma_{n}}P_{n}^{0}(\lambda-A^{m})^{-1}B(v)
  (\lambda-A^{m})^{-1}B(v)P_{n}^{0}\,d\lambda.
\end{equation}
As a $2\times2$ matrix,
\begin{equation*}
  \mathcal{S}_{n}=
  \begin{pmatrix}
    \mathcal{S}_{n}(2n-1,2n-1)  & \mathcal{S}_{n}(2n-1,-2n+1) \\
    \mathcal{S}_{n}(-2n+1,2n+1) & \mathcal{S}_{n}(-2n+1,-2n+1)
  \end{pmatrix}
\end{equation*}
is of the form
$$
  \mathcal{S}_{n}=K_{n}^{(2)}+
  \begin{pmatrix}
    0 & l(2(2n-1))\\
    l(-2(2n-1)) & 0
  \end{pmatrix}
,
$$
where $K_{n}^{(2)}$ is the diagonal part of $\mathcal{S}_{n}$.

Combining the computation above, one obtains the following
identity
\begin{equation}\label{eq_528}
  U_{n}^{-1}K_{n}U_{n}=
  \begin{pmatrix}
    0 & (v+l)(2(2n-1))\\
    (v+l)(-2(2n-1)) & 0
  \end{pmatrix}
   +L_{n}+R_{n},
\end{equation}
where $L_{n}$ is the diagonal matrix
$$L_{n}=L_{n}^{(1)}+L_{n}^{(2)}$$ and $R_{n}$ is the sum $$R_{n}=
\sum_{j=1}^{4}R_{n}^{(j)}.$$ The identity \eqref{eq_528} leads to
the following expression for the determinant
$$
  -\left(\frac{\gamma_{n}}{2}\right)^{2}=\det(U_{n}^{-1}K_{n}U_{n})=
  -(v+l)(2(2n-1))(v+l)(-2(2n-1))-r_{n},
$$
where the error $r_{n}$ is given by
\begin{align*}
   r_{n}&=-(K_{n}(2n-1,2n-1)+R_{n}(2n-1,2n-1))(K_{n}(-2n+1,-2n+1)+R_{n}(-2n+1,-2n+1)) \\
   & +(v+l)(2(2n-1))R_{n}(-2n+1,2n-1)+(v+l)(-2(2n-1))R_{n}(2n-1,-2n+1) \\
   & +R_{n}(-2n+1,2n-1)R_{n}(2n-1,-2n+1).
\end{align*}
Hence
$$
  \min_{\pm}\left|\frac{\gamma_{n}}{2}\pm\sqrt{(v+l)(-2(2n-1))
  \cdot(v+l)(2(2n-1))}\right|\leq|r_{n}|^{1/2}.
$$
To estimate $r_{n}$, use that an entry of a matrix is bounded by
its norm. Hence, for some universal constant $C>0$ and $n\geq
n_{*}$,
\begin{equation}\label{eq_530}
  |r_{n}|\leq C\left(\|K_{n}\|_{\mathcal{L}(h_{-}^{0})}^{2}+\|R_{n}\|_{\mathcal{L}(h_{-}^{0})}^{2}+\sum_{\pm}|
  (v+l)(\pm 2(2n-1))|\|R_{n}\|_{\mathcal{L}(h_{-}^{0})}\right).
\end{equation}
The terms on the right side of the inequality above are estimated
separately. By Theorem \ref{pr_18},
$$
  \|K_{n}^{(1)}\|_{\mathcal{L}(h_{-}^{0})}=|(2n-1)^{2m}\pi^{2m}-\tau_{n}|=h^{m(1-2\alpha-\varepsilon)}(n).
$$
As
\begin{equation*}
\|K_{n}^{(2)}\|=\emph{diag}\,(\mathcal{S}_{n}(2n-1,2n-1),
\mathcal{S}_{n}(-2n+1,-2n+1))
\end{equation*}
we have
$$\|K_{n}^{(2)}\|_{\mathcal{L}(h_{-}^{0})}\leq \|\mathcal{S}_{n}\|_{\mathcal{L}(h_{-}^{0})}$$
and, by the definition \eqref{eq_526} of $\mathcal{S}_{n}$,
\begin{align*}
  \|\mathcal{S}_{n}\|_{\mathcal{L}(h_{-}^{0})}=
  & \|\frac{1}{2\pi i}\int_{\Gamma_{n}}P_{n}^{0}A^{-m/2}_{\lambda}
  I_{\lambda}^{-1}S_{\lambda}I_{\lambda}^{-1}S_{\lambda}
  A^{m/2}_{\lambda}P_{n}^{0}\,d\lambda\| \\
  & \leq(2n-1)^{m}\left(\sup_{\lambda\in\Gamma_{n}}\|A^{-m/2}_{\lambda}\|_{\mathcal{L}(h_{-}^{0})}
    \|S_{\lambda}\|_{\mathcal{L}(h_{-}^{0})}^{2}\|A^{m/2}_{\lambda}P_{n}^{0}\|_{\mathcal{L}(h_{-}^{0})}\right).
\end{align*}
By Lemma \ref{l_26} and \eqref{eq_518} we then conclude
$$
  \|K_{n}^{(2)}\|_{\mathcal{L}(h_{-}^{0})}=h^{m(1-2\alpha-\varepsilon)}(n).
$$
By the definition of $R_{n}^{(1)}$,
\begin{equation*}
  \|R_{n}^{(1)}\|_{\mathcal{L}(h_{-}^{0})}\leq \|U_{n}^{-1}-P_{n}P_{n}^{0}\|_{\mathcal{L}(h_{-}^{0})}
  \|(A^{m}+B(v)-\tau_{n})P_{n}\|_{\mathcal{L}(h_{-}^{0})}\|P_{n}^{0}\|_{\mathcal{L}(h_{-}^{0})}.
\end{equation*}
We have $\|P_{n}^{0}\|_{\mathcal{L}(h_{-}^{0})}=1$ and
$$
  \|A^{m}+B(v)-\tau_{n})P_{n}\|_{\mathcal{L}(h_{-}^{0})}=\|\frac{1}{2\pi
  i}\int_{\Gamma_{n}}(\lambda-\tau_{n})(\lambda-A^{m}-B(v))^{-1}\,d\lambda\|_{\mathcal{L}(h_{-}^{0})}
  \leq C(2n-1)^{m\alpha},
$$
where for the last inequality we use Theorem \ref{pr_16} to deform
the contour $\Gamma_{n}$ to a circle $\Gamma_{n}^{'}$ of radius
$C(2n-1)^{m\alpha}$ around $(2n-1)^{2m}\pi^{2m}$ and the estimate
\begin{equation*}
\|(\lambda-A^{m}-B(v))^{-1}\|_{\mathcal{L}(h_{-}^{0})}=
\|A^{m/2}_{\lambda}(I_{\lambda}-S_{\lambda})^{-1}A^{m/2}_{\lambda}\|_{\mathcal{L}(h_{-}^{0})}\leq
C(2n-1)^{-m\alpha},\quad \forall\lambda\in\Gamma_{n}^{'}.
\end{equation*}
By the formula \eqref{eq_524} for $U_{n}^{-1}$, we have, in view
of the binomial formula \eqref{eq_522} and the definition
$Q_{n}=(P_{n}-P_{n}^{0})^{2}$,
\begin{equation*}
  \|U_{n}^{-1}-P_{n}P_{n}^{0}\|_{\mathcal{L}(h_{-}^{0})}\leq \|P_{n}\|_{\mathcal{L}(h_{-}^{0})}
  \|P_{n}^{0}\|_{\mathcal{L}(h_{-}^{0})}\sum_{l\geq 1}
  \left|
  \begin{pmatrix}
    -1/2 \\
      l
  \end{pmatrix}
  \right|
  \|P_{n}-P_{n}^{0}\|_{\mathcal{L}(h_{-}^{0})}^{2l},
\end{equation*}
where for the last inequality we use lemma \ref{l_32}(i) and the
estimate
\begin{equation*}
  \|P_{n}-P_{n}^{0}\|_{\mathcal{L}(h_{-}^{0})}\leq\frac{1}{2},\quad n\geq
  n_{*}.
\end{equation*}
Hence, by Lemma \ref{l_32}(ii),
$$
  \|R_{n}^{(1)}\|_{\mathcal{L}(h_{-}^{0})}\leq C(2n-1)^{-m\alpha}
  \|P_{n}-P_{n}^{0}\|_{\mathcal{L}(h_{-}^{0})}^{2}=
  (2n-1)^{m\alpha}(h^{m(1-\alpha-\varepsilon)}(n))^{2}.
$$
Similarly one shows
$$
  \|R_{n}^{(2)}\|_{\mathcal{L}(h_{-}^{0})}=(2n-1)^{m\alpha}(h^{m(1-\alpha-\varepsilon)}(n))^{2}.
$$
In view of the definition $R_{n}^{(3)}$,
$$
  \|R_{n}^{(3)}\|_{\mathcal{L}(h_{-}^{0})}\leq C\|P_{n}-P_{n}^{0}\|_{\mathcal{L}(h_{-}^{0})}
  |(2n-1)^{2m}\pi^{2m}-\tau_{n}|=(2n-1)^{m\alpha}(h^{m(1-\alpha-\varepsilon)}(n))^{2},
$$
where we use Lemma \ref{l_32} to estimate
$\|P_{n}-P_{n}^{0}\|_{\mathcal{L}(h_{-}^{0})}$ and Theorem
\ref{pr_18} to bound $$|(2n-1)^{2m}\pi^{2m}-\tau_{n}|.$$ Finally,
by by the definition of $R_{n}^{(4)}$ and Lemma \ref{l_26},
$$
  \|R_{n}^{(4)}\|_{\mathcal{L}(h_{-}^{0})}\leq C(2n-1)^{m}\|S_{\lambda}\|_{\mathcal{L}(h_{-}^{0})}^{3}\leq
  (2n-1)^{m}(h^{m(1-\alpha-\varepsilon/2)}(n))^{3}\leq
  (2n-1)^{m\alpha}(h^{m(1-\alpha-\varepsilon)}(n))^{2}.
$$
Combining the obtained estimates one gets
\begin{equation}\label{eq_532}
  \|R_{n}\|_{\mathcal{L}(h_{-}^{0})}=(2n-1)^{m\alpha}(h^{m(1-\alpha-\varepsilon)}(n))^{2},
\end{equation}
\begin{equation}\label{eq_534}
  \|K_{n}\|=h^{m(1-2\alpha-\varepsilon)}(n).
\end{equation}
Taking to account that
\begin{equation*}
\|l\|_{h_{+}^{mt}}\leq Const\|w\|_{h_{+}^{mt}},\quad t>-m\alpha,
\end{equation*}
as a consequence we have obtained
\begin{equation}\label{eq_536}
  |(v+l)(\pm 2(2n-1))|\|R_{n}\|_{\mathcal{L}(h_{-}^{0})}=(h^{m(1-2\alpha-\varepsilon)}(n))^{2}
\end{equation}
and, in view of \eqref{eq_530},
$$
  |r_{n}|^{1/2}=h^{m(1-2\alpha-\varepsilon)}(n).
$$
This proves Theorem \ref{pr_20}.
\end{proof}

\section{The limiting case $\alpha=1$} In this Section the
spectral problem
\begin{equation*}
  T u=\lambda u
\end{equation*}
for the operator
\begin{equation*}
  T(v)\equiv T_{1}(v)=A^{m}+B(v),\quad v\in h_{+}^{-m}
\end{equation*}
is studied.

At first, in a straightforward way, one can prove the following
two auxiliary lemmas.
\begin{lemma}\label{l_34}
For any $s,t\in\mathbb{R}$ with $s-t\leq 2$ and any
\begin{equation*}
  \lambda\in\mathbb{C}\setminus spec(A^{m}),\quad m\in\mathbb{N}
\end{equation*}
we have
$$(\lambda-A^{m})^{-1}\in\mathcal{L}(h_{-}^{mt},h_{-}^{ms})$$ with
norm
\begin{equation*}\label{eq_538}
  \parallel(\lambda-A^{m})^{-1}\parallel_{\mathcal{L}(h_{-}^{mt},h_{-}^{ms})}=
  \sup_{k\in\mathbb{Z}}\frac{\langle 2k-1\rangle^{m(s-t)}}{|\lambda-(2k-1)^{2m}\pi^{2m}|}<\infty.
\end{equation*}
\end{lemma}
\begin{lemma}\label{l_36}
Uniformly for $n\in\mathbb{Z}\setminus\{0\}$ and $\lambda\in
Vert^{m}_{n}(r_{n})$ the following estimates are valid:
\begin{align*}
  (a)& \parallel(\lambda-A^{m})^{-1}\parallel_{\mathcal{L}(h_{-}^{-m})}=\frac{1}{r_{n}}O(1),
  & \quad (a') & \parallel(\lambda-A^{m})^{-1}\parallel_{\mathcal{L}(h_{-}^{-m})}=O(n^{-m}), \\
  (b) & \parallel(\lambda-A^{m})^{-1}\parallel_{\mathcal{L}(h_{-}^{-m,n})}=\frac{1}{r_{n}}O(1),
  & \quad
  (b')&\parallel(\lambda-A^{m})^{-1}\parallel_{\mathcal{L}(h_{-}^{-m,n})}=O(n^{-m}),\\
  (c)& \parallel(\lambda-A^{m})^{-1}\parallel_{\mathcal{L}(h_{-}^{-m,n},h_{-}^{-m})}=\frac{1}{r_{n}}O(1),
  & \quad (c') &
  \parallel(\lambda-A^{m})^{-1}\parallel_{\mathcal{L}(h_{-}^{-m,n},h_{-}^{-m})}=O(n^{-m}),
  \\
  (d)& \parallel(\lambda-A^{m})^{-1}\parallel_{\mathcal{L}(h_{-}^{-m},h_{-}^{m,n})}=\frac{(2n-1)^{2m}}{r_{n}}O(1),
  & \quad (d') &
  \parallel(\lambda-A^{m})^{-1}\parallel_{\mathcal{L}(h_{-}^{-m},h_{-}^{m,n})}=O(n^{m}),
  \\
  (e)&
  \parallel(\lambda-A^{m})^{-1}\parallel_{\mathcal{L}(h_{-}^{-m,n},h_{-}^{m,-n})}=\frac{(2n-1)^{m}}{r_{n}}O(1);
  & \quad (e') &
  \parallel(\lambda-A^{m})^{-1}\parallel_{\mathcal{L}(h_{-}^{-m,n},h_{-}^{m,-n})}=O(1).
\end{align*}
\end{lemma}
\begin{theorem}\label{pr_22}
Let $m\in\mathbb{N}$, and $v\in h_{+}^{-m}$. There exist
$\varepsilon>0$, $M\geq 1$ and $n_{0}\in \mathbb{N}$ so that for
any $w\in h_{+}^{-m}$ with $$\parallel
w-v\parallel_{h_{+}^{-m}}\leq\varepsilon$$ the spectrum
$\emph{spec}(T(w))$ of the operator $$T(w)=A^{m}+B(w)$$  consists
of a sequence $(\lambda_{k}(m,w))_{k\geq 1}$ such that:
\begin{itemize}
  \item [(a)] There are precisely $2n_0$ eigenvalues
  inside the bounded cone
\begin{equation*}\label{eq_540}
  T_{M,n_0}=\left\{\lambda\in\mathbb{C}\,\left|\,\right.|Im\,\lambda|-M\leq Re\,\lambda \leq
    ((2n)_{0}^{2m}-(2n)_{0}^{m})\pi^{2m}\right\}.
\end{equation*}
  \item [(b)] For $n> n_0$ the pairs of eigenvalues
  $\lambda_{2n-1}(m,w)$, $\lambda_{2n}(m,w)$ are inside a
  disc around \\ $(2n-1)^{2m}\pi^{2m}$:
\begin{align*}
  |\lambda_{2n-1}(m,w)-(2n-1)^{2m}\pi^{2m}| & <(2n-1)^{m}, \\
  |\lambda_{2n}(m,w)-(2n-1)^{2m}\pi^{2m}| & <(2n-1)^{m}.
\end{align*}
\end{itemize}
\end{theorem}
\begin{proof}
Let $v\in h_{+}^{-m}$. Since the set $h_{+}^{m}$ is dense in the
space $h_{+}^{-m}$, we can represent $v$ in the form
\begin{equation*}
v=v_{0}+v_{1},\quad \text{with}\quad v_{0}\in h_{+}^{m}\quad
\text{and}\quad
\parallel v_{1}\parallel_{h_{+}^{-m}}\leq\varepsilon,
\end{equation*}
where $\varepsilon>0$ will be find bellow. We will show that for
some $M\geq 1$ and $n_{0}\in\mathbb{N}$, which are both depending
on $\parallel v_{0}\parallel_{h_{+}^{m}},$ so that for any
$w=v+\tilde{w}\in h_{+}^{-m}$ with $\parallel
\tilde{w}\parallel_{h_{+}^{-m}}\leq\varepsilon$, we have got
\begin{equation}\label{eq_542}
  Ext_{M}\cup\bigcup_{n\geq n_{0}}Vert^{m}_{n}((2n-1)^{m})\subseteq\emph{Resol}
  (T(w)),
\end{equation}
where $\emph{Resol}(T(w))$ denotes the resolvent set of the
operator $$T(w)=A^{m}+B(v_{0})+B(v_{1}+\tilde{w}).$$ At first let
consider $\lambda\in Ext_{M}$ for $M\geq 1$. Using the Convolution
Lemma and the Lemma \ref{l_34} one gets
\begin{equation*}\label{eq_544}
  \parallel
  B(v_{0})(\lambda-A^{m})^{-1}\parallel_{\mathcal{L}(h_{-}^{-m})}
  \leq C_{m}\parallel v_{0}\parallel_{h_{+}^{m}}
  \parallel(\lambda-A^{m})^{-1}\parallel_{\mathcal{L}(h_{-}^{-m})} =\parallel
  v_{0}\parallel_{h_{+}^{m}}\cdot O(M^{-1}).
\end{equation*}
Hence, for $M\geq 1$ large enough and $\lambda\in Ext_{M}$,
\begin{equation*}\label{eq_546}
  T_{\lambda}:=\lambda-A^{m}-B(v_{0})
  =(Id-B(v_{0})(\lambda-A^{m})^{-1}) (\lambda-A^{m})
\end{equation*}
is invertible in $\mathcal{L}(h_{-}^{-m})$ with inverse
\begin{equation}\label{eq_548}
  T_{\lambda}^{-1}=(\lambda-A^{m})^{-1}
  (Id-B(v_{0})(\lambda-A^{m})^{-1})^{-1}.
\end{equation}
So, using the Convolution Lemma and the estimate
$$\parallel(\lambda-A^{m})^{-1}\parallel_{\mathcal{L}
(h_{-}^{-m},h_{-}^{m})}=O(1),$$ we have obtained
\begin{equation*}\label{eq_550}
  \parallel B(v_{1}+\tilde{w})T_{\lambda}^{-1} \parallel_{\mathcal{L}(h_{-}^{-m})}
  =O(\parallel(v_{1}+\tilde{w})\parallel_{h_{+}^{-m}})=O(\varepsilon).
\end{equation*}
Therefore, if $\varepsilon> 0$ is small enough, the resolvent of
the operator $$T(w)=A^{m}+B(v_{0})+B(v_{1}+\tilde{w})$$ exists in
the space $\mathcal{L}(h_{-}^{-m})$ for $\lambda\in Ext_{M}$ and
is given by the formula
\begin{equation}\label{eq_552}
  (\lambda-A^{m}-B(v_{0})-B(v_{1}+\tilde{w}))^{-1}=(T_{\lambda}-B(v_{1}+\tilde{w}))^{-1}
  =T_{\lambda}^{-1}\sum_{k\geq 0}(B(v_{1}+\tilde{w})T_{\lambda}^{-1})^{k}.
\end{equation}
Consequently, for $M$ large enough,
\begin{equation*}
   Ext_{M}\subseteq \emph{Resol}(T(w)).
\end{equation*}

To treat $\lambda\in Vert^{m}_{n}((2n-1)^{m})$, first note that,
unfortunately, $$\parallel(\lambda-A^{m})^{-1}
\parallel_{\mathcal{L}(h_{-}^{-m},h_{-}^{m})}=O(n^{m}),$$ and so
we can not argue as above. However, we have (see the Lemma
\ref{l_36} $(e')$)
\begin{equation}\label{eq_554}
\parallel(\lambda-A^{m})^{-1}
\parallel_{\mathcal{L}(h_{-}^{-m,n},h_{-}^{m,-n})}=O(1).
\end{equation}
Now, for $\lambda\in Vert^{m}_{n}((2n-1)^{m})$ with $n$ large
enough, we find that the following decomposition of the resolvent
of $$T(w)=A^{m}+B(v_{0})+B(v_{1}+\tilde{w})$$ converges in the
space $\mathcal{L}(h_{-}^{-m})$,
\begin{equation}\label{eq_556}
  (\lambda-A^{m}-B(v_{0})-B(v_{1}+\tilde{w}))^{-1}=T_{\lambda}^{-1}
  +T_{\lambda}^{-1}K_{\lambda} (B(v_{1}+\tilde{w})T_{\lambda}^{-1}) +T_{\lambda}^{-1}K_{\lambda}
  (B(v_{1}+\tilde{w})T_{\lambda}^{-1})^{2},
\end{equation}
where $$T_{\lambda}=\lambda-A^{m}-B(v_{0}),$$ and
$$K_{\lambda}:=\sum_{l\geq 0}
(B(v_{1}+\tilde{w})T_{\lambda}^{-1})^{2l}$$ is considered as an
element in $\mathcal{L}(h_{-}^{-m,n})$. Using the Convolution
Lemma ($c'$) and the Lemma \ref{l_36} ($a'$), ($b'$) we can find
$n_{0}\in\mathbb{N}$ such that, for any $n\geq n_{0}$ and
$\lambda\in Vert^{m}_{n}((2n-1)^{m})$, the operator $T_{\lambda}$
is invertible in the spaces $\mathcal{L}(h_{-}^{-m})$ and
$\mathcal{L}(h_{-}^{-m,n})$ in the form \eqref{eq_548}. Using the
Convolution Lemma ($a'$) and the Lemma \ref{l_36} ($e'$), one can
obtain
\begin{equation*}\label{eq_558}
  \parallel B(v_{1}+\tilde{w})T_{\lambda}^{-1}
  \parallel_{\mathcal{L}(h_{-}^{-m,n},h_{-}^{-m,-n})}\leq C_{m}\parallel
  (v_{1}+\tilde{w})\parallel_{h_{+}^{-m}}
  \parallel T_{\lambda}^{-1}\parallel_{\mathcal{L}
  (h_{-}^{-m,n},h_{-}^{m,-n})}=O(\varepsilon).
\end{equation*}
Therefore, if $\varepsilon>0$ is small enough, the sum
$$K_{\lambda}=\sum_{l\geq 0}
(B(v_{1}+\tilde{w})T_{\lambda}^{-1})^{2l}$$ converges in
$\mathcal{L}(h_{-}^{-m,n})$. Then the representation
\eqref{eq_556} follows because
$$B(v_{1}+\tilde{w})T_{\lambda}^{-1}\in
\mathcal{L}(h_{-}^{-m},h_{-}^{-m,n})$$ by the Convolution Lemma
($a'$) and the Lemma \ref{l_36} ($d'$), and $$T_{\lambda}^{-1}\in
\mathcal{L}(h_{-}^{-m,n},h_{-}^{-m})$$ by the Lemma \ref{l_36}
($c'$).

Hence, for $\varepsilon> 0$, $M\geq 1$, and $n_{0}\in \mathbb{N}$
as above, the inclusion \eqref{eq_542} holds. Let remark, that in
fact, we have proved the inclusion
\begin{equation}\label{eq_560}
  Ext_{M}\cup\bigcup_{n\geq
  n_{0}}Vert^{m}_{n}((2n-1)^{m})\subseteq\emph{Resol}
  (T(w(s))),
\end{equation}
where $\emph{Resol}(T(w(s)))$ denotes the resolvent set of the
operator
\begin{equation*}
  T(w(s))=A^{m}+B(v_{0})+s B(v_{1}+\tilde{w})\quad \text{for}\quad 0\leq s\leq 1.
\end{equation*}
Hence, for any contour
\begin{equation*}
\Gamma\subset Ext_{M}\cup\bigcup_{n\geq
n_{0}}Vert^{m}_{n}((2n-1)^{m}),
\end{equation*}
and any $0\leq s\leq 1$, the Riesz projector
\begin{equation*}\label{eq_562}
  P(s):=\frac{1}{2\pi i}\int_{\Gamma}
  (\lambda-A^{m}-B(v_{0})-sB(v_{1}+\tilde{w}))^{-1}\,d\lambda\in
  \mathcal{L}(h_{-}^{-m}),
\end{equation*}
is well defined and depends continuously on $s$. Since projectors
whose difference has a small norm have isomorphic ranges a
continuity of the map
\begin{equation*}
  s\mapsto P(s)
\end{equation*}
implies that the dimension of the range $P(s)$ is independent of
$s$. Therefore the number of eigenvalues of the operators
$$A^{m}+B(v_{0})$$ and $$A^{m}+B(v_{0})+B(v_{1}+\tilde{w})$$
are the same (counted with their algebraic multiplicity) inside
$\Gamma$. To complete the proof of Theorem \ref{pr_22} it is
sufficient to apply Theorem \ref{pr_16} to the operator
$$A^{m}+B(v_{0})$$ with
$$v_{0}\in h^{m}\subseteq h^{0}.$$
\end{proof}
\begin{theorem}\label{pr_24} Let $v$ in $h_{+}^{-m}$, and
$R\geq 0$. For any $w\in h_{+}^{-m}$ with $$\parallel
w-v\parallel_{h_{+}^{m}}\leq R$$ the spectrum
$\emph{spec}(A^{m}+B(w))$ of the operator $T(w)=A^{m}+B(w)$
consists of a sequence \\ $(\lambda_{k}(m,w))_{k\geq 1}$ of
eigenvalues and the following uniform in $w$ asymptotic formulae
\begin{align*}
  \lambda_{2n-1}(m,w) & =(2n-1)^{2m}\pi^{2m}+o(n^{m}),\quad n\rightarrow\infty, \\
  \lambda_{2n}(m,w) & =(2n-1)^{2m}\pi^{2m}+o(n^{m}),\quad
  n\rightarrow\infty
\end{align*}
are hold.
\end{theorem}
\begin{proof}
Let $v\in h_{+}^{-m}$. Since the set $h_{+}^{m}$ is dense in the
space $h_{+}^{-m}$, one decomposes $$v=v_{0}+v_{1},$$ with
\begin{equation*}
v_{0}\in h_{+}^{m},\quad \text{and} \quad\parallel
v_{1}\parallel_{h_{+}^{-m}}\leq\varepsilon,
\end{equation*}
where $\varepsilon>0$ will be chosen bellow. We are going to show
as above that there exists $n_{0}\in\mathbb{N}$ depending on
$$\parallel v_{0}\parallel _{h_{+}^{m}}$$ and $R\geq 0$ such that, for
any $w=v+w_{0}\in h_{+}^{-m}$ with $\parallel
w_{0}\parallel_{h_{+}^{-m}} \leq R$,
\begin{equation}\label{eq_564}
  \bigcup_{n\geq n_{0}}Vert^{m}_{n}((2n-1)^{m}) \subseteq\emph{Resol}
  (T(w)),
\end{equation}
where $\emph{Resol}(T(w))$ denotes the resolvent set of the
operator $$T(w)=A^{m}+B(v_{0}+w_{0})+B(v_{1}).$$ Notice, that now
we consider the strips $Vert^{m}_{n}(r_{n})$ with $$r_{n}=\delta
(2n-1)^{m}$$ for some $\delta\in (0,1]$.

So, let $\lambda\in Vert^{m}_{n}(r_{n})$. Using the Convolution
Lemma and the Lemma \ref{l_36} (b) one gets
\begin{equation*}\label{eq_566}
  \parallel
  B(v_{0}+w_{0})(\lambda-A^{m})^{-1}\parallel_{\mathcal{L}(h_{-}^{-m,n})}
  \leq C_{m}\parallel (v_{0}+w_{0})\parallel_{h_{+}^{m}}
  \parallel(\lambda-A^{m})^{-1}\parallel_{\mathcal{L}(h_{-}^{-m,n})} =\frac{\parallel
  (v_{0}+w_{0})\parallel_{h_{+}^{m}}}{r_{n}}O(1).
\end{equation*}
Hence, for $n$ large enough and $\lambda\in Vert^{m}_{n}(r_{n})$,
\begin{equation*}\label{eq_568}
  T_{\lambda}:=\lambda-A^{m}-B(v_{0}+w_{0})
  =(Id-B(v_{0}+w_{0})(\lambda-A^{m})^{-1}) (\lambda-A^{m})
\end{equation*}
is invertible in $\mathcal{L}(h_{-}^{-m})$ with inverse
\begin{equation}\label{eq_570}
  T_{\lambda}^{-1}=(\lambda-A^{m})^{-1}
  (Id-B(v_{0}+w_{0})(\lambda-A^{m})^{-1})^{-1}.
\end{equation}
Further, for $n$ large enough, we can show that the following
representation  of  resolvent of the operator $$T(w)=
A^{m}+B(v_{0}+w_{0})+B(v_{1})$$ converges in
$\mathcal{L}(h_{-}^{-m})$,
\begin{equation}\label{eq_572}
  (\lambda-A^{m}-B(v_{0}+w_{0})-B(v_{1}))^{-1}=T_{\lambda}^{-1}
  +T_{\lambda}^{-1}K_{\lambda} (B(v_{1})T_{\lambda}^{-1}) +T_{\lambda}^{-1}K_{\lambda}
  (B(v_{1})T_{\lambda}^{-1})^{2},
\end{equation}
where $$T_{\lambda}=\lambda-A^{m}-B(v_{0}+w_{0}),$$ and
$$K_{\lambda}:=\sum_{l\geq 0} (B(v_{1})T_{\lambda}^{-1})^{2l}$$ is
considered as an element in $\mathcal{L}(h_{-}^{-m,n})$. Using the
Convolution Lemma  and the Lemma \ref{l_36} (e), we get
\begin{equation*}\label{eq_574}
  \parallel B(v_{1})T_{\lambda}^{-1}
  \parallel_{\mathcal{L}(h_{-}^{-m,n},h_{-}^{-m,-n})}\leq C_{m}\parallel
  v_{1}\parallel_{h_{+}^{-m}} \parallel T_{\lambda}^{-1}\parallel _{\mathcal{L}
  (h_{-}^{-m,n},h_{-}^{m,-n})}=O(\varepsilon).
\end{equation*}
Hence, if $\varepsilon>0$ is small enough, the sum
$$K_{\lambda}=\sum_{l\geq 0} (B(v_{1})T_{\lambda}^{-1})^{2l}$$
converges in the space $\mathcal{L}(h_{-}^{-m,n})$ and the
representation \eqref{eq_572} then follows because
$$B(v_{1})T_{\lambda}^{-1}\in
\mathcal{L}(h_{-}^{-m},h_{-}^{-m,n})$$ by the Convolution Lemma
and the Lemma \ref{l_36} (d), and $$T_{\lambda}^{-1}\in
\mathcal{L} (h_{-}^{-m,n},h_{-}^{-m})$$ by the Lemma \ref{l_36}
(c).

Consequently, for some $\varepsilon> 0$ and $n_{0}\in \mathbb{N}$
he inclusion \eqref{eq_564} holds for $$r_{n}=\delta n^{m},\quad
\delta\in (0,1].$$

So, for any contour
\begin{equation*}
\Gamma\subset \bigcup_{n\geq n_{0}}Vert^{m}_{n}((2n-1)^{m}),
\end{equation*}
and any $0\leq s\leq 1$, the Riesz projector
\begin{equation*}\label{eq 576}
  P(s):=\frac{1}{2\pi i}\int_{\Gamma}
  (\lambda-A^{m}-B(v_{0}+w_{0})-B(s v_{1}))^{-1}\,d\lambda\in
  \mathcal{L}(h_{-}^{-m}).
\end{equation*}
is well defined and depends continuously on $s$. Since projectors
whose difference has a small norm have isomorphic ranges
continuity of the map
\begin{equation*}
  s\mapsto P(s)
\end{equation*}
implies that the dimension of the range $P(s)$ is independent of
$s$. Therefore the number of eigenvalues of the operators
$$A^{m}+B(v_{0}+w_{0})$$ and $$A^{m}+B(v_{0}+w_{0})+B(v_{1})$$ inside $\Gamma$
(counted with their algebraic multiplicity) are the same. Applying
Theorem \ref{pr_16} to the operator $$A^{m}+B(v_{0}+w_{0})$$ one
gets:

the spectrum $\emph{spec}(T(w))$ of the operator
$$T(w)=A^{m}+B(w)$$ consists of a
sequence $(\lambda_{k}(m,w))_{k\geq 1}$ of complex-valued
eigenvalues, and for any $\delta\in (0,1]$ there exists
$n_{0}\in\mathbb{N}$ such that the pairs of eigenvalues
$\lambda_{2n-1}(m,w)$, $\lambda_{2n}(m,w)$ there are inside a disc
around $(2n-1)^{2m}\pi^{2m}$,
\begin{align*}
  |\lambda_{2n-1}(m,w)-(2n-1)^{2m}\pi^{2m}| & <\delta (2n-1)^{m}, \\
  |\lambda_{2n}(m,w)-(2n-1)^{2m}\pi^{2m}| & <\delta (2n-1)^{m}.
\end{align*}

So, we conclude that the sequence $$(\lambda_{k}(m,w))_{k\geq 1}$$
of eigenvalues satisfies the asymptotic formulae
\begin{align*}
  \lambda_{2n-1}(m,w) & =(2n-1)^{2m}\pi^{2m}+o(n^{m}),\quad n\rightarrow\infty, \\
  \lambda_{2n}(m,w) & =(2n-1)^{2m}\pi^{2m}+o(n^{m}),\quad
  n\rightarrow\infty.
\end{align*}
The proof is complete.
\end{proof}

\section{Acnowledgement}
The first author (V.A.M.) was partially supported by NFBR of
Ukraine under Grants 01.07/027 and 01.07/00252.

\end{document}